\documentclass[abstracton, DIV=14, paper=a4, fontsize=11pt,bibliography=totoc, final]{scrartcl}

\pdfoutput=1

\usepackage[utf8]{inputenc}
\usepackage[T1]{fontenc}
\usepackage{lmodern}
\usepackage[english]{babel}
\usepackage{csquotes}
\usepackage{amsthm, amssymb, amsmath, amsfonts, mathrsfs, dsfont, esint,textcomp}
\usepackage[colorlinks=true, pdfstartview=FitV, linkcolor=blue, citecolor=blue, urlcolor=blue,pagebackref=false]{hyperref}
\usepackage{mathtools}
\usepackage[shortlabels]{enumitem}

\usepackage[backend=biber,giveninits,style=alphabetic,maxnames=4,isbn=false, doi=false, eprint= true, url= false, sorting=anyt, sortcites = true]{biblatex}

\usepackage{tikz}
\usetikzlibrary{decorations.pathreplacing}
\usetikzlibrary{arrows}
\usetikzlibrary{hobby}

\usepackage{microtype}

\usepackage[notcite,notref,color]{showkeys}
\definecolor{labelkey}{gray}{.8}
\definecolor{refkey}{gray}{.8}

\definecolor{darkblue}{rgb}{0,0,0.7} 
\definecolor{darkred}{rgb}{0.9,0.1,0.1}
\definecolor{darkgreen}{rgb}{0,0.5,0}

\renewenvironment{proof}[1][\smallskip\noindent\proofname]{{\smallskip\noindent\bfseries #1. }}{\qed}

\newtheorem{thm}{Theorem}[section]

\newtheorem{prop}[thm]{Proposition}
\newtheorem{lem}[thm]{Lemma}

\theoremstyle{remark}
\newtheorem{rem}[thm]{Remark}
\theoremstyle{definition}

\renewcommand{\leq}{\leqslant}
\renewcommand{\geq}{\geqslant}

\renewcommand{\subset}{\subseteq}

\newcommand{\1}{\mathbf{1}}
\newcommand{\R}{\mathbb{R}}

\newcommand{\Z}{\mathbb{Z}}

\newcommand{\eps}{\varepsilon}

\newcommand{\dd}{\, \mathrm{d}}
\newcommand{\wto}{\rightharpoonup}

\newcommand{\loc}{\mathrm{loc}}

\DeclareMathOperator{\Id}{Id}

\DeclareMathOperator*{\dv}{div}

\DeclareMathOperator*{\supp}{supp}

\numberwithin{equation}{section}

\mathtoolsset{showonlyrefs}

\title{Homogenization of the Navier-Stokes equations in perforated domains in the inviscid limit}

\author{Richard M. H\"ofer\thanks{richard.hoefer@ur.de, Faculty of Mathematics, University of Regensburg, Germany}}

\begin{document}

\maketitle

\begin{abstract}
We study the solution $u_\eps$ to the Navier-Stokes equations in $\R^3$  perforated by small particles centered at $(\eps \Z)^3$ with no-slip boundary conditions at the particles.
We study the behavior of $u_\eps$  for small $\eps$, depending on the diameter $\eps^\alpha$, $\alpha > 1$, of the particles and the viscosity $\eps^\gamma$, $\gamma > 0$, of the fluid. We prove quantitative convergence results for $u_\eps$ in all regimes when the local Reynolds number at the particles is negligible. Then, the particles approximately exert a linear friction force on the fluid. The obtained  effective macroscopic equations depend on the order of magnitude of the collective friction. We obtain a) the Euler-Brinkman equations in the critical regime,  b) the Euler equations in the subcritical regime and c) Darcy's law in the supercritical regime.
\end{abstract}

{\small 

\smallskip

}

\section{Introduction}

The homogenization of fluid flows in perforated domains has been intensively studied in the last decades. Various models for the fluids reaching from incompressible inviscid flows (see e.g. \cite{HillairetLacaveWu22, MikelicPaoli99,LacaveFilhoLopes18,LacaveMasmoudi16}) to compressible viscous flows (see e.g. \cite{Masmoudi02, HoeferKowalczykSchwarzacher21,BellaOschmann23, Oschmann22}) and even non-Newtonian fluids (see e.g. \cite{Mikelic18})   have been considered with different boundary conditions, including Navier slip conditions (see e.g. \cite{Allaire91}) and so-called sedimentation boundary conditions (see e.g. \cite{NiethammerSchubert19, Gerard-VaretHoefer21,DuerinckxGloria21Viscosity}).

From the application oriented point of view, interest in such homogenization problems arises from the study of flow through porous media and of suspension flows. In the case of such particulate flows, homogenization problems where the particle evolution is frozen or prescribed can be considered as a first step towards the derivation of fully coupled models between the fluid flow and the dispersed phase.

The limiting behavior of solutions to the incompressible 
(Navier-)Stokes equations with fixed viscosity in perforated domains with no-slip boundary conditions is by now  quite well understood. On the microscopic lengthscale of the particles, the fluid inertia becomes negligible. Therefore, in the limit of many small particles, a linear friction relation (Stokes law) prevails, giving rise to an effective massive term, the so-called Brinkman term. Depending on the particle sizes and number density, the Brinkman term becomes negligible, dominant or of order one in the homogenization limit, leading to the (Navier-)Stokes equations, Darcy's law and the (Navier-)Stokes-Brinkman, respectively, see e.g. \cite{Tartar80, Allaire90a, Allaire90b, Mikelic91,  DesvillettesGolseRicci08, Feireisl2016, HillairetMoussaSueur19, GiuntiHoefer19,  CarrapatosoHillairet20, Giunti21, HoeferJansen20, LuYang23}.

For  the case of the Navier-Stokes equations with vanishing viscosity, only very few results are available though. 
The problem of considering such fluids in perforated domains with very small viscosity (or more precisely large macroscopic Reynolds numbers) is a very relevant one in applications.
Indeed, in the modeling of sprays, it is not unusual to couple  kinetic equations for the dispersed phase to the Euler equations (see e.g. \cite{BarangerDesvillettes06, CarrilloDuanMoussa11}).
%
On the other hand, regarding porous media, understainding flow at large Reynolds number is  very important (see e.g. \cite{BalhoffMikelicWheeler10}) and nonlinear extensions of Darcy's law, in particular the  Darcy-Forchheimer equations, are proposed at very large Reynolds numbers. 
Although the rigorous derivation of such \emph{nonlinear} effective models seems currently out of reach, the present work aims at identifying the effective behavior in all scaling limits where a linear friction law prevails.
We emphasize that the effective models we obtain are completely different from the ones that result by starting from the Euler equations in perforated domains (see e.g. \cite{LacaveMasmoudi16, MikelicPaoli99, LacaveFilhoLopes18, HillairetLacaveWu22} for such models). 
Instead, correspondingly to the (Navier-)Stokes equations with constant viscosity, we identify and prove homogenization limits in a critical, subcritical and supercritical  regime yielding the Euler-Brinkman equations, the Euler equations and Darcy's law, respectively. 
To the author's knowledge, the Euler-Brinkman equations have not even been formally derived in the literature before. This can be viewed as a first step towards the rigorous justification of spray models like the one analyzed in \cite{CarrilloDuanMoussa11} that couples the incompressible Euler equations to a Vlasov equation through a linear friction force.

\subsection{Setting and outline of the main results}

Let $\mathcal T \Subset  B_{1/4}(0)$, 
the reference particle, be a fixed closed set with smooth boundary, such that $B_{1}(0) \setminus \mathcal T$ is connected and $0 \in \mathring  {\mathcal T}$.
For $0 < \eps < 1$, we consider particles centered at $x^\eps_i := \eps i$,  $i \in \Z^3$. Moreover, precisely, for $\alpha \geq 1$, we define
\begin{align}
 \Omega_\eps := \R^3 \setminus \cup_{i \in \Z^3} \mathcal T_i^\eps, &&
    \mathcal T_i^\eps := x_i^\eps + \eps^{\alpha} \mathcal T
\end{align}

Then, for some $T>0$, $\gamma > 0$ and $\mu_0 > 0$, we consider solutions $u_\eps$ to the Navier-Stokes equations
\begin{align} \label{Navier.Stokes.holes}
\begin{array}{rcll}
    \partial_t u_\eps + u_\eps \cdot \nabla u_\eps- \mu_0 \eps^\gamma \Delta u_\eps + \nabla p_\eps &=& f_\eps &\qquad \text{in } (0,T) \times \Omega_\eps, \\
        \dv u_\eps &=& 0 &\qquad \text{in }  (0,T) \times \Omega_\eps, \\
    u_\eps &=& 0 &\qquad \text{on }  (0,T) \times \partial \Omega_\eps, \\
    u_\eps(0,\cdot) &=& u_0^\eps &\qquad \text{in } \Omega_\eps
    \end{array}
\end{align}
for some given $f_\eps \in L^2(0,T;L^2(\R^3))$  and $u_0^\eps \in L^2_\sigma(\Omega_\eps)$,  where
\begin{align}
    L^2_\sigma(\Omega_\eps) := \{ v \in L^2(\Omega_\eps) : \dv v = 0,  v \cdot n = 0 \text{ on } \partial \Omega_\eps \}.
\end{align}
It is well known
that then at least one  Leray  solution $u_\eps$ exist, i.e. a weak solution which satisfies the energy inequality
\begin{align} \label{energy.inequality}
    \frac 1 2 \|u_\eps(t)\|^2_{L^2(\Omega_\eps)} + \mu_0 \eps^\gamma \|\nabla u_\eps\|^2_{L^2((0,t) \times \Omega_\eps)} \leq \frac 1 2 \|u_0^\eps\|_{L^2(\Omega_\eps)}^2 + \int_0^t \int_{\Omega_\eps} f_\eps \cdot u_\eps \dd x \dd t \quad \forall 0 \leq t \leq T. \quad
\end{align}

We focus on the case $\alpha > 1$ which characterizes the regime where the particle diameters $\eps^\alpha$ are small compared to the inter-particle distance $\eps$. 
In a nutshell, the effect of the particles on the fluid can then be described through a superposition of linear friction laws provided that the fluid inertia is negligible on the lengthscale of the particles. 
More precisely, we consider the particle Reynolds number
\begin{align} \label{particle.Reynolds}
	\mathrm{Re}^\eps_{\mathrm{part}} := \frac{\text{particle diameter}\times\text{fluid velocity}}{\text{viscosity}} =  U_\eps {\eps^{\alpha-\gamma}} 
\end{align} 
where $U_\eps$, the order of magnitude of the fluid velocity,  has yet to be determined. 
Then, if $\mathrm{Re}^\eps_{\mathrm{part}} \ll 1$, the influence of each particle on the fluid can be approximated by a friction force determined from  the unique solutions $(w_k,q_k) \in \dot H^1(\R^3) \times L^2(\R^3)$ to the linear Stokes problem
\begin{align} \label{w_k}
\begin{array}{rcll}
    - \Delta w_k + \nabla q_k &=& 0 &\qquad \text{in } \R^3 \setminus \mathcal T, \\
    \dv w_k &=& 0&\qquad \text{in } \R^3 \setminus \mathcal T, \\
    w_k &=& e_k& \qquad \text{on } \partial \mathcal T
\end{array}
\end{align}
through the associated resistance matrix $\mathcal R \in \R^{3 \times 3}$
\begin{align} \label{def:R}
    \mathcal R_{jk} = \int_{\R^3 \setminus \mathcal T} \nabla w_k : \nabla w_j,
\end{align}
which is a positive definite symmetric matrix.
Neglecting fluid inertia and particle interaction,  classical scaling considerations imply that each particle approximately contributes a friction force $F_i = - \mu_0 \eps^{\alpha+\gamma} \mathcal R (u_\eps)_i$
where $(u_\eps)_i$ should be understood as a suitable average of $u_\eps$ on some lenthscale $\eps^\alpha \ll d_\eps \leq \eps$ around $x_i^\eps$. Taking into account that the particle number density is $\eps^{-3}$ leads to approximating the fluid velocity $u_\eps$ by $\tilde u_\eps$ which satisfies the Navier-Stokes equations in the \emph{whole space} with an additional linear friction term $\mu_0 \eps^{\alpha + \gamma - 3} \mathcal R \tilde u_\eps$, sometimes referred to as \emph{Brinkman force}. More precisely, provided $\mathrm{Re}^\eps_{\mathrm{part}} \ll 1$, we expect $u_\eps \approx \tilde u_\eps$ where 
\begin{align} \label{Navier.Stokes.approx}
\begin{array}{rcll}
    \partial_t \tilde u_\eps + \tilde u_\eps \cdot \nabla \tilde u_\eps- \mu_0 \eps^\gamma \Delta \tilde u_\eps + \mu_0 \eps^{\alpha + \gamma - 3} \mathcal R \tilde u_\eps + \nabla \tilde p_\eps &=& f_\eps &\qquad \text{in } (0,T) \times \R^3, \\
        \dv u_\eps &=& 0 &\qquad \text{in }  (0,T) \times \R^3.
    \end{array}
\end{align}
 
From this approximation, we may easily identify the limiting behavior, where we distinguish the \emph{critical} regime as $\gamma + \alpha = 3$, the \emph{subcritical} regime as $\gamma + \alpha > 3$ and the supercritical regime as $\gamma + \alpha < 3$. 
Before writing down the limiting equations, we revisit the constraint $\mathrm{Re}^\eps_{\mathrm{part}} \ll 1$. 
In the critical and subcritical regime, the Brinkman force is at most of order one, and therefore the solution $\tilde u_\eps$, and thus $u_\eps$ and $U_\eps$ from \eqref{particle.Reynolds}, are expected to be of order $1$, provided $u_0^\eps$ and $f_\eps$ are of order $1$.
Thus, in the critical and subcritical regime,
\begin{align}
\mathrm{Re}^\eps_{\mathrm{part}} = \frac{\eps^{\alpha - \gamma}}{\mu_0},
\end{align} 
which leads to the condition $\alpha > \gamma$.

On the other hand, in the supercritical regime, the Brinkman force dominates thus slows down the fluid velocity to
$U_\eps = \eps^{3-\alpha - \gamma}$.
Therefore, in the supercritical case,
\begin{align}
\mathrm{Re}^\eps_{\mathrm{part}} = \frac{\eps^{3 - 2\gamma}}{\mu_0},
\end{align} 
leading to the condition $\gamma < 3/2$.

\medskip

Taking the formal limit in \eqref{Navier.Stokes.approx}, assuming $f_\eps \to f$ and $u_0^\eps \to u_0$ leads to the following limit systems. The regimes are illustrated in Figure \ref{fig.regimes}.

\begin{figure}[ht] \label{fig.regimes}
\begin{center} \includegraphics[scale=0.8]{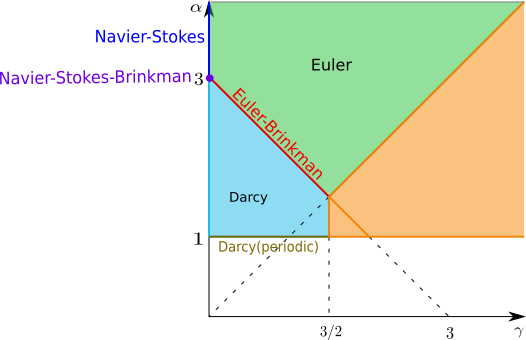} \hspace*{4cm}
\end{center}
\caption{Scaling regimes of effective equations}
\end{figure}

\begin{itemize}
\item In the critical regime $\alpha + \gamma = 3$ with $\alpha > 1$ and $\alpha > \gamma > 0$, we obtain (for $\mu_0 = 1$ for simplicity) the Euler-Brinkman equations\footnote{One might argue that Euler-Darcy would be a more appropriate name for this system but this is already used for a different system that arises as  homogenization limit of the $2$-dimensional Euler equations in perforated domains, see e.g. \cite{MikelicPaoli99}.}
\begin{align} \label{Euler-Darcy}
    \begin{array}{rcll}
    \partial_t u + u \cdot \nabla u + \mathcal R u + \nabla p &=& f &\qquad \text{in } (0,T) \times \R^3, \\
    \dv u &=& 0  &\qquad \text{in } (0,T) \times \R^3, \\
    u(0,\cdot) &=& u_0 &\qquad \text{in } \R^3.
    \end{array}
\end{align}

\item In the subcritical regime for   $ \alpha + \gamma > 3$ with $\alpha > 1$ and $\alpha > \gamma > 0$, we obtain the Euler equations
\begin{align} \label{Euler}
    \begin{array}{rcll} \partial_t u + u \cdot \nabla u +\nabla p &=& f &\qquad \text{in } (0,T) \times \R^3, \\
    \dv u &=& 0  &\qquad \text{in } (0,T) \times \R^3, \\
    u(0,\cdot) &= &u_0 &\qquad \text{in } \R^3.
    \end{array}
\end{align}
Since the particles do not create any effective perturbation on the limit system, the asymptotically linear friction law guaranteed by $\alpha > \gamma > 0$ is actually not required to obtain this limit case but it instead suffices that $\mathrm{Re}^\eps_{\mathrm{part}} \leq c_0$ for some $c_0 > 0$ independent of $\eps$.  This corresponds to the regime  $\alpha = \gamma > 3/2 $
with $\mu_0 \geq M$ for some $M$ sufficiently large.

\item In the supercritical regime, for $\alpha + \gamma < 3$ with $\alpha > 1$ and $\gamma < 3/2$,  $u_\eps \to 0$. Thus, we rescale time and velocities to obtain a nontrivial limit. More precisely, if $\hat u_\eps$ is a solution to \eqref{Navier.Stokes.holes} with $\mu_0 = 1$, we consider the function $ u_\eps(t,x) =  \eps^{\alpha + \gamma-3} \hat u_\eps(\eps^{\alpha+\gamma-3} t,x)$. This 
rescaled velocity solves (after rescaling accordingly $f_\eps$, $p_\eps$ and $u_\eps^0$ without renaming them)
\begin{align} \label{Navier.Stokes.holes.rescaled}
\begin{array}{rcll}
     \eps^{6 - 2\alpha-2\gamma} \left(\partial_t u_\eps + u_\eps \cdot \nabla u_\eps \right)-  \eps^{3-\alpha} \Delta u_\eps + \nabla p_\eps &=& f_\eps & \qquad \text{in } (0,T) \times \Omega_\eps, \\
    \dv u_\eps &=& 0 &\qquad \text{in } \Omega_\eps, \\
    u_\eps &=& 0 &\qquad \text{on } \partial (0,T) \times  \Omega_\eps, \\
    u_\eps(0,\cdot) &=& u_0^\eps &\qquad \text{in } \Omega_\eps.
\end{array}
\end{align}
Performing the same rescaling on the system \eqref{Navier.Stokes.approx}, we formally obtain Darcy's law in the limit $\eps \to 0$, namely
\begin{align} \label{Darcy}
\begin{array}{rcll}
    \mathcal R u + \nabla p &=& f &\qquad \text{in } (0,T) \times \R^3,    \\
     \dv u  &=&0 &\qquad \text{in } (0,T) \times  \R^3.
\end{array}
\end{align}
\end{itemize}

\subsection{Statement of the main results}
The precise results are the following quantitative convergence results for $u_\eps$  in all three regimes under regularity assumption on the solution $u$ to the respective limit system. Smooth solutions exist at least for short times. 
Moreover, in the supercritical regime, we obtain in addition a weak convergence result
in $L^2(0,T;L^2(\R^3))$ assuming only a weak solution $u \in L^2(0,T;L^2(\R^3))$ to Darcy's law \eqref{Darcy}.

\begin{thm}[Critical regime] \label{th:Euler-Brinkman}
      Let $\alpha \in (3/2,3)$, $\gamma = 3 -\alpha$ and $\mu_0 = 1$. Let $T>0$, $u_0 \in H^4(\R^3)$, $f \in C(0,T;H^2(\R^3))$ and $(u,p) \in C^1(0,T;H^4(\R^3)) \times L^\infty(0,T;H^3_\loc(\R^3))$ be a solution to \eqref{Euler-Darcy}. Moreover, for $0 < \eps < 1$ let $u^\eps_0 \in L^2_\sigma(\Omega_\eps)$, $f_\eps \in L^2(0,T;L^2(\Omega_\eps))$
      and let $u_\eps \in L^2(H^1_0(\Omega_\eps)) \cap C(0,T;L^2(\Omega_\eps))$ be  a Leray solution to \eqref{Navier.Stokes.holes}. Then, there exists   $C > 0$ which depends only on  the reference particle $\mathcal T$ and, monotonously, on $T$, $\|f\|_{L^\infty(0,T;H^2(\R^3))}$,  $\|u\|_{C^1(0,T;H^4(\R^3))}$ and $\|\nabla p\|_{L^\infty(0,T;H^2(\R^3))}$ such that for all $t \leq T$ 
  \begin{align} \label{error-estimate}
      \|( u_\eps - u)(t)\|_{L^2(\Omega_\eps)}^2\leq  C \left(\|u_\eps^0 - u_0\|_{L^2(\Omega_\eps)}^2 + \|f_\eps - f\|_{L^2(0,T;L^2(\Omega_\eps))}^2 +  \left(\eps^{2 \alpha - 3 } +  \eps^{6-2\alpha}\right) \right) .
  \end{align}
\end{thm}

\begin{thm}[Subcritical regime] \label{th:Euler}
      Let $\mu_0 > 0$, $\alpha > 3/2, \gamma > 0$ satisfy $3 - \alpha < \gamma \leq \alpha$. Let $T>0$, $u_0 \in H^4(\R^3)$, $f \in C(0,T;H^2(\R^3))$ and $(u,p) \in C^1(0,T;H^4(\R^3)) \times L^\infty(0,T;H^3_\loc(\R^3))$ be a solution to \eqref{Euler}. Moreover, for $0 < \eps <1$ let $u^\eps_0 \in L^2_\sigma(\Omega_\eps)$, $f_\eps \in L^2(0,T;L^2(\Omega_\eps))$
      and let $u_\eps \in L^2(H^1_0(\Omega_\eps)) \cap C(0,T;L^2(\Omega_\eps))$ be  a Leray solution to \eqref{Navier.Stokes.holes}. Then, there exists $M>0$ depends only on the reference particle $\mathcal T$ and, monotonously, on  $T$,  $\|f\|_{L^\infty(0,T;H^2(\R^3))}$, $\|u\|_{C^1(0,T;H^4(\R^3))}$ and $\|\nabla p\|_{L^\infty(0,T;H^2(\R^3))}$, and $C > 0$ which depends additionally on $\mu_0$ such that, if either $\alpha > \gamma$ or $\mu_0 \geq M$, we have for all $t \leq T$ 
  \begin{align} \label{error-estimate.Euler}
       \|(u_\eps - u)(t)\|_{L^2(\Omega_\eps)}^2 &\leq  C \left(\|u_\eps^0 - u_0\|_{L^2(\Omega_\eps)}^2 + \|f_\eps - f\|_{L^2(0,T;L^2(\Omega_\eps))}^2 
      + \left(\eps^{2 \alpha + 2\gamma-6} + \eps^{2 \alpha  -3}    + \eps^{2\gamma} \right) \right).
  \end{align}
\end{thm} 
In the supercritical regime, we remind that we consider the rescaled system \eqref{Navier.Stokes.holes.rescaled}. The corresponding energy inequality reads

\begin{align} \label{energy.inequality.Darcy}
    \frac 1 2 \|u_\eps(t)\|^2_{L^2(\Omega_\eps)} + \eps^{2\gamma + \alpha -3}  \|\nabla u_\eps\|^2_{L^2(0,t;L^2(\Omega_\eps))} \leq \frac 1 2 \|u_0^\eps\|_{L^2(\Omega_\eps)}^2   +\eps^{2 \gamma + 2 \alpha -6}  \int_0^t \int_{\Omega_\eps} f_\eps \cdot u_\eps \dd x \dd s \qquad 
\end{align}
for all $0 \leq t \leq T$.

\begin{thm}[Supercritical regime -- quantitative result] \label{th:Darcy.quantitative}
         Let $\alpha \in (1,3)$ and $0 < \gamma < \min\{3/2,3-\alpha\}$. Let $T >0$ and $f \in C^1(0,T;H^4(\R^3))$ and let $(u,p) \in C^1(0,T;H^4(\R^3)) \times C^1(0,T;H^5_\loc(\R^3)) $ be the unique solution to \eqref{Darcy} (up to constants for the pressure).
     For $\eps > 0$ let $u^\eps_0 \in L^2_\sigma(\Omega_\eps)$  and let $u_\eps$ be a 
 Leray solutions to \eqref{Navier.Stokes.holes.rescaled}.
       Then, there exists $C > 0$ which depends only on  the reference particle $\mathcal T$ and, monotonously, on  $T$,  $\|f\|_{L^\infty(0,T;H^2(\R^3))}$, $\|u\|_{C^1(0,T;H^4(\R^3))}$, $\|\nabla p\|_{L^\infty(0,T;H^2(\R^3))}$ and $\|u_\eps^0\|_{L^2(\Omega_\eps)}$ such that for all $t \leq T$ 
      \begin{align}
                  \|u_\eps - u\|_{L^2((0,T) \times \Omega_\eps)}^2  &\leq C\left( \eps^{6-2\alpha-2 \gamma} \|u_\eps^0 - u_0\|^2_{L^2(\Omega_\eps)} + \|f_\eps - f\|_{L^2((0,T) \times \Omega_\eps)}^2 \right.  \\
                  &\qquad \qquad \qquad \qquad \left.+ \eps^{\frac{6 - 4\gamma}{3}} + \eps^{\alpha - 1} + \eps^{9-3\alpha} + \eps^{12 - 4 \alpha -4 \gamma}\right).
    \end{align}
\end{thm}
\begin{rem}
\begin{itemize}
    \item The three theorems above imply in particular that for any sequence $\eps \to 0$ with $\|u_\eps^0 - u_0\|_{L^2(\R^3)} \to 0$ (respectively $\eps^{6-2\alpha-2 \gamma} \|u_\eps^0 - u_0\|^2_{L^2(\Omega_\eps)} \to 0$), and $f_\eps \to f$ in $L^2(0,T;L^2(\R^3))$ we have $u_\eps \to u$ in $L^\infty(0,T;L^2(\R^3))$ (respectively in $L^2(0,T;L^2(\R^3))$). Here, $f_\eps$, $u_\eps^0$ and $u_\eps$  are to be understood as defined in $\R^3$ through extension by $0$.
    Note that one may choose $f_\eps = f$ in $\Omega_\eps$. 
    Moreover, one may choose $u_\eps^0 = w^\eps u_0$ with $w^\eps$ as in Section $\ref{sec:w}$. Then, estimate \eqref{w_eps-1} guarantees $\|u_\eps^0 - u_0\|_{L^2(\R^3)} \to 0$ for any choice of the parameter $\eps^\alpha \leq \eta_\eps \leq \eps$ that $w^\eps$ depends on. Optimizing $\eta_\eps$ yields $\|u_\eps^0 - u_0\|^2_{L^2(\R^3)} \leq C \eps^{3\alpha -3}$.
    
    
    \item The regularity assumptions on $u$ could probably be weakened but we do not pursue to optimize here.
    
    \item In the supercritical regime, we do not obtain pointwise estimates in time. Indeed, there are boundary layers in time  which prevent pointwise estimates under the stated assumptions. These boundary layers are due to the initial datum $u_\eps$ but also due to possible jumps in time of the force $f_\eps$. 
\end{itemize}
\end{rem}

\begin{thm} [Supercritical regime -- qualitative result] \label{th:Darcy}
      Assume  $\alpha \in (1,3)$, $0 < \gamma < \min\{3/2,3-\alpha\}$ .
     For $T>0$ and  $\eps > 0$, assume $u^\eps_0 \in L^2(\Omega_\eps)$ such that $ \eps^{3 - \alpha -\gamma} \|u_0^\eps\|_{L^2(\Omega_\eps)}$ is uniformly bounded and $f_\eps \in L^2(0,T;L^2(\R^3))$ converges weakly to some $f$ in $L^2(0,T;L^2(\R^3))$. Let  $u_\eps \in L^2(H^1_0(\Omega_\eps)) \cap C(0,T;L^2(\Omega_\eps))$ be  a Leray solution to \eqref{Navier.Stokes.holes.rescaled}. Then, $\tilde u_\eps \wto u$ in $L^2(0,T;L^2(\R^3))$, where  $u$ is the unique weak solution in $L^2(0,T,L^2(\R^3))$ to \eqref{Darcy} and where $\tilde u_\eps$ is the extensions of $u_\eps$ to $\R^3$ by $\tilde u_\eps = 0$ in $\R^3 \setminus \Omega_\eps$.
     \end{thm}

\subsection{Previous results}

The vanishing viscosity limit is a classical problem in the study of incompressible fluids, we refer to \cite{MaekawaMazzucato18} for a review on the topic.
In bounded domains with no-slip boundary conditions, the limiting behavior is not well-understood due to the onset of boundary layers.
This is the reason why we consider the whole space in this paper.
 
In dimensions two and three, the vanishing viscosity limit has been studied in \cite{IftimieFilhoLopes09} in the presence of a single shrinking body. The convergence to the Euler equations has been established provided that the local Reynolds number is sufficiently small i.e. the same condition $a_\eps \leq  c \mu_\eps \ll 1$, where $a_\eps$ and $\mu_\eps$ denote the particle diameter and fluid viscosity, respectively, and $c$ is a sufficiently small constant (depending on the initial data, time, and the reference particle). 

\medskip 

There is a vast literature on homogenization in perforated domains. Modeling the fluid velocity $u_\eps$ by the stationary Stokes equations, Darcy's law has been obtained in
\cite{Tartar80} in the case of particle of the same size as the inter-particle distance, i.e. $\alpha =1$.
Later, Allaire \cite{Allaire90a,Allaire90b}
proved homogenization results for the Stokes equations for all ranges of $\alpha > 1$, identifying Darcy's law for $\alpha \in (1,3)$, the Stokes-Brinkman equations for $\alpha =3$ and the Stokes equations for $\alpha > 3$.
Allaire's results cover all space dimensions $d \geq 2$ with appropriate adaptations of the ranges of $\alpha$ for $d \geq 4$. In the two-dimensional case, the critical regime corresponds to particle diameters $a_\eps$ such that $\eps^{-2} \log a_\eps \sim 1$.
By compactness, Allaire's results also apply to the stationary Navier-Stokes equations (in dimensions $d \leq 4$).

The results of Allaire have been refined in a number of works, for example considering more general distributions of particles, non-homogeneous Dirichlet boundary conditions, the study of higher order approximations and fluctuations. We refer to the recent results \cite{DesvillettesGolseRicci08, HillairetMoussaSueur19, GiuntiHoefer19, CarrapatosoHillairet20, HoeferJansen20, Giunti21} and the references therein.

\medskip

The homogenization limits for the full instationary Navier-Stokes for \emph{fixed viscosity} correspond to the one of the stationary Stokes equations and are displayed in Figure \ref{fig.regimes}.
Formally they are obtained by setting $\gamma=0$ in \eqref{Navier.Stokes.approx} and taking the limit $\eps \to 0$. 
The critical regime, $\alpha =3$, leading to  the Navier-Stokes-Brinkman equations, has  been considered by Feireisl, Ne\v{c}asov\'a and Namlyeyeva in \cite{Feireisl2016}, whereas the subcritical case $\alpha > 3$ and the supercritical case $\alpha \in (1,3)$ has been treated recently by Lu and Yang in \cite{LuYang23}. 

The case $\alpha = 1$, including the full range of vanishing viscosities $\gamma \in [0,3/2)$ has been treated by Mikeli\'c \cite{Mikelic91}.

We emphasize that the Darcy's law in \cite{LuYang23} and \cite{Allaire90b} is  exactly the same as \eqref{Darcy} whereas the Darcy's law in \cite{Tartar80} and \cite{Mikelic91} differs quantitatively, in terms of a different resistance tensor $\mathcal R_{\mathrm{per}}$ which is obtained analogously as $\mathcal R$ from \eqref{def:R} but by solving the Stokes equations in the torus instead of the whole space.
The reason for this difference is that in the case $\alpha = 1$ the particle diameter is comparable to the interparticle distance. Therefore, the superposition of friction forces through single particle problems in the whole space (cf. \eqref{w_k}) must be replaced by studying the collective forces through the problem with periodic boundary conditions. Mathematically, the analysis of the case $\alpha =1$ is somewhat easier as it only involves two lengthscales, the microscopic lengthscale $\eps$ and the macroscopic lengthscale. Since the study of the case $\alpha =1$ requires different corrector problems and is rather well understood, we restrict our attention to $\alpha > 1$ in the present paper.

\medskip

Reflecting its importance for applications, there are several works concerning the derivation of non-linear Darcy's laws, especially the Darcy-Forchheimer equations.
They seem to focus on the case $\alpha =1$, where nonlinear effects are expected to become important for $\gamma \geq 3/2$. 
Most of these works do not contain rigorous proofs, we refer to \cite{BalhoffMikelicWheeler10} for an overview of the literature. 
Concerning rigorous results,  Mikeli\'c \cite{Mikelic95} and Maru{\v{s}}i{\'c}-Paloka and Mikeli\'c \cite{Marusic-PalokaMikelic00} tackled the critical case $\alpha = 1$, $\gamma = 3/2$  in dimensions two and three starting from the stationary Navier-Stokes equations. 
 The obtained limit system is a nonlinear nonlocal Darcy type equation. 
Moreover, in the subcritical case, $\alpha=1$, $\gamma < 3/2$, Bourgeat, Maru{\v{s}}i{\'c}-Paloka and Mikeli\'c \cite{BourgeatMarusicPalokaMikelic95} justified nonlinear versions of Darcy's law as higher order corrections to the linear law.

\smallskip

We also mention that the homogenization of the instationary Stokes equations with vanishing viscosity has been studied by Allaire \cite{Allaire92} for $\alpha =1$. In this case, the critical scaling (in any space dimension) is $\gamma =2$ and a Darcy's law with memory effect is obtained as limit system.

\medskip
The only previous result the author is aware of concerning the homogenization of the Navier-Stokes equations with vanishing viscosities when the particle diameters are much smaller than the interparticle distance ($\alpha > 1$) is due to Lacave and Mazzucato \cite{LacaveMazzucato16}.  
In dimension two, they recover the unperturbed Euler equations under assumptions on the particle sizes, distances and the viscosity, which guarantuee that the particle Reynolds number is sufficiently small and that  the particles do not exert a significant collective force on the fluid (subcritical regime).

\medskip

\subsection{Elements of the proof} \label{sec:strategy}

The proof of the (quantitative) main results is based on an energy argument to estimate $u_\eps - u$ which is, at its core, classical in the study of vanishing viscosity limits.  However, similarly as in \cite{IftimieFilhoLopes09} and \cite{LacaveMazzucato16},  we face the problem, that the limit fluid velocity $u$ does not vanish inside of the particles and thus $u$ is not an admissible testfunction for the PDE of $u_\eps$.
As in  \cite{IftimieFilhoLopes09} and \cite{LacaveMazzucato16}, we therefore consider functions $\hat u_\eps$ obtained from $u$ by a suitable truncation.
In \cite{IftimieFilhoLopes09}, the truncation is performed on the level of the stream function (respectively the vector potential in three dimensions). In \cite{LacaveMazzucato16}, 
the fluid velocity itself is truncated, i.e.
\begin{align}
	\hat u_\eps = \phi_\eps u + h_\eps,
\end{align}
where $h_\eps$ is a suitable Bogovskii type correction  such that $\hat u_\eps$ is divergence free.

As in \cite{LacaveMazzucato16}, we perform the truncation on the level of the fluid velocity itself. However, we need to be more careful, since the truncation needs to contain information of the boundary layers at the particles that produce the Brinkman term in the limit.
Thus, instead of the scalar function $\phi_\eps$ in \cite{LacaveMazzucato16} that truncate in a $\eps^\alpha$ neighborhood around the particles,  we choose a variant of the matrix-valued oscillating testfunction $w^\eps$ used by Allaire \cite{Allaire90a,Allaire90b} that are build on the solutions to the resistance problem \eqref{w_k}.

These functions $w^\eps$ from  \cite{Allaire90a,Allaire90b} (which go back to corresponding functions in \cite{Tartar80} and similar functions for the Poisson equations used by Cioranescu and Murat in \cite{CioranescuMurat82a}) have been used  with some modifications in many related works, see e.g. \cite{GiuntiHoefer19, LuYang23}.
However, $w^\eps$ truncates on an $\eps$-neighborhood around the particles, and therefore we could only use them directly in the present context provided the Reynolds number on the $\eps$-lengthscale is small. This is the case if
$\gamma < 1$ in the (sub-)critical regime and $\gamma <2 - \alpha/2$ in the supercritical regime. 
To overcome this restriction, we modify the 
testfunctions of Allaire, to truncate on a lengthscale $\eta_\eps$, $\eps^\alpha \leq \eta_\eps \leq \eps$. 
Aside from estimates analogous to their standard versions, we then use a Hardy-type estimate in order to control some error-terms arising from the nonlinear convection term. 

\subsection{Some possible generalizations and open problems}

In this paper, we focus on  periodic distributions of identical particles
for the sake of the clarity of the presentation.
The methods of proof do not rely on periodicity, though, and presumably apply to more general settings.

From the viewpoint of applications to suspensions, it would also be interesting to study  non-homogeneous Dirichlet boundary conditions, i.e. $u_\eps = V_i$ on $\partial T^\eps_i$ which have been treated for the corresponding model without vanishing viscosity in \cite{DesvillettesGolseRicci08, Feireisl2016}. 

\medskip

As in many related works, we focus here on the three-dimensional case. Extensions to two dimensions are possible with the necessary modifications similar as in \cite{Allaire90a, Allaire90b}. As mentioned above, parts of the subcritical regime
is treated in \cite{LacaveMazzucato16}. There is one important difference between the two- and three-dimensional case, however, that seems to make it more difficult to analyze all the cases in dimensions two where the particle Reynolds number tends to zero. Namely, in three dimensions, the Stokes resistance of a particle of size $a_\eps$ in the whole space is well approximated by solving Stokes problems in an $\eta_\eps$-neighborhood of the particle, for any  lengthscale $\eta_\eps$ with $\eta_\eps \gg a_\eps$. This allows us to consider the intermediate scale $\eta_\eps$ as outlined in the previous subsection. 
In two dimensions, however, just like for capacities, only \emph{relative} Stokes resistances are meaningful. As observed in \cite{Allaire90a, Allaire90b}, it turns out that the relative resistance in a cell of order of the inter-particle distance $\eps$  is the correct object to consider in order to study the collective effect of the particles.\footnote{To be more precise, since the relative Stokes resistance scales like $|\log(\eta_\eps/a_\eps)|^{-1}$ in two dimensions,  it does not matter whether one chooses $\eta_\eps = \eps$ or $\eta_\eps = \eps^\beta$. However, one should allow $a_\eps$ to be much smaller than powers of $\eps$ in order to include the critical case $-\eps^2 \log a_\eps  \sim 1$.}
Therefore, the  use of an intermediate lengthscale $\eta_\eps$ does not seem suitable in $2$ dimensions, at least not in the critical and supercritical regimes. As discussed above, this would restrict  to assuming that the Reynolds number on the scale $\eps$ is of order one, in order that the (accordingly modified) proof given in this paper still works.

\medskip

It would be of great interest to understand the regimes where the particle Reynolds number $\mathrm{Re}^\eps_{\mathrm{part}}$ is not tending to zero, i.e. $\gamma \geq \max\{\alpha,3/2\}$, displayed in orange in Figure \ref{fig.regimes}.
However, as discussed above, the case when the particle Reynolds number is large is not even understood in the case of  a single shrinking particle.
In the case where the particle Reynolds number is small but fixed, we proved that one still obtains the Euler equations in the subcritical regime. One could still expect convergence to the Euler equations in the subcritical regime. In the critical and supercritical regimes,  one could expect the onset of nonlinear behavior similar to the one obtained in \cite{Mikelic95,Marusic-PalokaMikelic00} at $\gamma = 3/2$.

\subsection{Outline of the rest of the paper}

The rest of the paper is organized as follows.

In Section \ref{sec:w}, we define the correctors $w^\eps$ and prove some useful estimates on them. Mostly, these are standard adaptions of previously established estimates.

Section \ref{sec:proof.main} contains the proofs of the main results. In Section \ref{sec:proof.subcritical} we give the proofs of Theorem \ref{th:Euler-Brinkman} and Theorem \ref{th:Euler}, which are largely analogous.

Section \ref{sec:Darcy}  contains the proof of Theorem \ref{th:Darcy} and Theorem  \ref{th:Darcy.quantitative}. 
The proof of Theorem \ref{th:Darcy.quantitative} is very similar to those of Theorems \ref{th:Euler-Brinkman} and \ref{th:Euler}. For the proof of Theorem \ref{th:Darcy},  we first use a well-known Poincaré inequality in the perforated domain (see Proposition \ref{prop:Poincare}) to get a uniform a-priori estimate of $u_\eps$ in $L^2(0,T;L^2(\R^3))$. We use a classical duality argument that allows us to pass to the limit in the weak formulation of the PDE by applying the correctors $w^\eps$ to smooth  testfunctions instead of the solution $u$ of the limit problem as in the proof of the quantitative results.

\section{Corrector estimates} \label{sec:w}

Throughout this section, we write $A \lesssim B$ for $A,B \in \R$ when $A \leq C B$ for some constant $C$ that depends only on the reference particle $\mathcal T$ and possibly the exponent $p$ of some Sobolev space involved in the estimate.

\begin{figure}[ht] 
 \begin{center}
\begin{tikzpicture} [scale = 0.08]
\filldraw[fill = orange!35] (-37.8,-38.5) rectangle ++(80,80);
\filldraw[fill = red!35] (2.2,1.5) circle (30);
\filldraw[fill = blue!35] (2.2,1.5) circle (15);
\filldraw[fill={rgb:black,1;white,2}] (0,0) to [closed, curve through = {(1.2,2) (0.4,3.3) (1.9,4.8)  (3.9,4.2)(3.2,2.4) (4,0) (3,-1.4) (0.3,-1.2)}] (0,0);
 
\draw (2.4, 4.9) -- (45.0, 4.85);
 \draw (2.0, -1.8) -- (45.0, -1.75);
 \draw [<-] (45.0, 4.85) -- (45.0, 4);
  \draw [->] (45.0, -0.7) -- (45.0, -1.75);
 \node (C) at (45.0,1.7) {\Large $\varepsilon^{\alpha}$};
 
 \draw (2.2, 31.5) -- (58.0, 31.5);
 \draw (2.2, -28.5) -- (58.0, -28.5);
  \draw [<-] (58.0, 31.5) -- (58.0, 5);
 \draw [->] (58.0, -1.6) -- (58.0, -28.5);
 \node (D) at (58.0,1.7) {\Large $\eta_\eps$};
 
 \draw [<-] (52.0, 16.5) -- (52.0, 6);
  \draw [->] (52.0, -2.7) -- (52.0, -13.5);
 \draw (2.2, 16.5) -- (52.0, 16.5);
 \draw (2.2, -13.5) -- (52.0, -13.5);
 \node (E ) at (52.0,1.7) {\Large $\frac{\eta_\eps}2$};
 
 \draw (2.2, 41.5) -- (63.0, 41.5);
 \draw (2.2, -38.5) -- (63.0, -38.5);
  \draw [<-] (63.0, 41.5) -- (63.0, 5);
 \draw [->] (63.0, -1.6) -- (63.0, -38.5);
 \node (D) at (63.0,1.7) {\Large $\varepsilon$};
 
 \filldraw[fill = orange!35] (68,-26) rectangle ++(5,8);
 \filldraw[fill = red!35] (68,-16) rectangle ++(5,8);
 \filldraw[fill = blue!35] (68,-6) rectangle ++(5,8);
 \filldraw[fill={rgb:black,1;white,2}] (68,4) rectangle ++(5,8);
 
 \node (F) at (77.0,-23) {\Large  $K_i^{\varepsilon}$};
 \node (G) at (77.0,-13) {{\Large $D_i^{\varepsilon}$}};
 \node  (H) at (77.0,-3) {{\Large $C_i^{\varepsilon}$}};
 \node(I) at (77.0,7) {{\Large $\mathcal T_i^{\varepsilon}$}};

\end{tikzpicture}
\end{center}
\caption{Decomposition of cell the ${Q}_{i}^{\eps}$}
\label{figure cell decomposition} 
\end{figure}
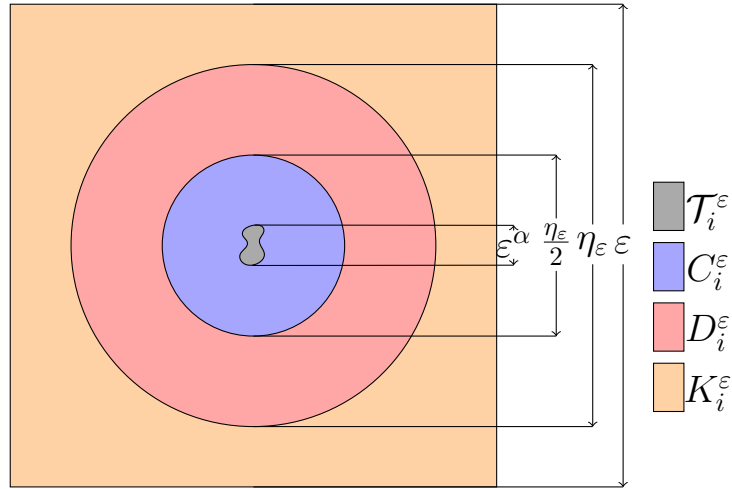

Let $\eps^\alpha \leq \eta_\eps \leq \eps$.
  We denote by $Q_i^\eps$ the open cubes of length $\eps$ centered at $x_i^\eps$ that (essentially) cover $\R^3$. We split each cube $Q_i^\eps$  into four areas,
  displayed in Figure \ref{figure cell decomposition},
\begin{align}
    Q_i^\eps &= \mathcal T_i^\eps \cup C_i^\eps \cup D_i^\eps \cup K_i^\eps,  \\
    C_i^\eps &:= B_{\frac{\eta_\eps}4}(x_i^\eps) \setminus \mathcal T_i^\eps,  \\
    D_i^\eps &:= B_{\frac{\eta_\eps}2}(x_i^\eps) \setminus B_{\frac{\eta_\eps}4}(x_i^\eps), \\
    K_i^\eps &:= Q_i^\eps \setminus B_{\frac{\eta_\eps}2}(x_i^\eps).
\end{align}
 Then, recalling the definition of $(w_k,q_k)$ from \eqref{w_k},  we define $w^\eps_k,q_k^\eps$ as the $\eps$-periodic functions
that satisfy $(w_k^\eps,q_k^\eps) \in W^{1,\infty}_0(\Omega_\eps) \times L^\infty(\Omega_\eps)$, and, in $Q_i^\eps$
\begin{align} \label{w^eps}
 \begin{array}{rl}
    w_{k}^\eps(x) = e_k - w_k\left(\frac{x-x_i^\eps}{\eps^\alpha}\right), \quad q_k^\eps(x) = - \eps^{-\alpha} q_k\left(\frac{x-x_i^\eps}{\eps^\alpha}\right) &\qquad \text{in } C_i^\eps, \\
    - \Delta w_{k}^\eps(x) + \nabla q^\eps = 0,  \quad \dv w_k^\eps = 0 &\qquad \text{in } D_i^\eps, \\
    w_k^\eps = e_k, \quad  q_k^\eps = 0 &\qquad \text{in } K_i^\eps. 
\end{array}
\end{align}
Here, $e_k$ denotes the $k$-th unit vector of the standard basis of $\R^3$. Note that the Stokes equations in $D_i^\eps$ are complemented with inhomogeneous no slip boundary conditions due to the requirement $w_k^\eps \in W^{1,\infty}_0(\Omega_\eps)$.
We will write $w^\eps$ for the matrix-valued function with columns $w_k^\eps$, and $q^\eps$ for the (row-)vector with entries $q_k^\eps$.
We summarize properties of $w^\eps$ in the following lemmas. Some of the estimates are very similar to the ones given in \cite{Allaire90a,Allaire90b} and other works.

\begin{lem} \label{lem:w^eps}
    The functions $w^\eps$, $q^\eps$ satisfy
    \begin{enumerate}[(i)]
    \item \label{item:infty} $w^\eps \in W^{1,\infty}_0(\Omega_\eps)$, $q^\eps \in L^\infty(\Omega_\eps)$, $\dv w^\eps_k = 0$ for $k=1,2,3$ and 
    \begin{align} \label{nabla.w.infty}
   \|w^\eps\|_{L^\infty(\R^3)} + \eps^{\alpha} \left( \|\nabla w^\eps\|_{L^\infty(\R^3)} + \|q^\eps\|_{L^\infty(\R^3)}  \right) \lesssim 1.
\end{align}

\item \label{item:L^2}  For all compact sets $K \subset \R^3$, we have $w^\eps \to \Id$ strongly in $L^2(K)$. 
Moreover, for all $3/2 < p < 3$ and all $\varphi \in W^{2,p}(\R^3)$
\begin{align} \label{w_eps-1}
    \|\varphi (\Id- w^\eps)\|_{L^p(\R^3)} \lesssim  \eta_\eps^{\frac 3 p - 1} \eps^{\alpha   - \frac{3}p} \|\varphi\|_{W^{2,p}(\R^3)}.
\end{align}
Furthermore, 
\begin{align} \label{w_eps-1.p=3}
     \|\varphi (\Id- w^\eps)\|_{L^3(\R^3)} \lesssim \eps^{\alpha - 1} |\log \eps|^{\frac 13 } \|\varphi\|_{W^{2,3}(\R^3)} &\qquad \text{for all } \varphi \in W^{2,3}(\R^3), 
    \\
      \|\varphi \nabla w^\eps \|_{L^2(\R^3)} +  \|\varphi q^\eps \|_{L^2(\R^3)} \lesssim \eps^{\frac{\alpha -3}{2}} \|\varphi\|_{H^2(\R^3)} &\qquad \text{for all } \varphi \in H^2(\R^3), \label{wk.qk.phi} \\
      \||\nabla w^\eps|^{\frac 1 2}\varphi  \|_{L^2(\R^3)} + \| |q^\eps|^{\frac 1 2} \varphi \|_{L^2(\R^3)} \lesssim \eta_\eps^{\frac 1 2}  \eps^{\frac{\alpha -3}{2}} \|\varphi\|_{H^2(\R^3)} &\qquad \text{for all } \varphi \in H^2(\R^3) \label{wk.qk.phi.1}. 
\end{align}

\item \label{item:Hardy} For all $\varphi \in H_0^1(\Omega_\eps)$
\begin{align} \label{Hardy}
    \||\nabla w^\eps|^{\frac 1 2}\varphi\|_{L^2(\Omega_\eps)} + | q^\eps|^{\frac 1 2}\varphi\|_{L^2(\Omega^\eps)}\lesssim \eta_\eps^{\frac 12 }  \| \nabla \varphi\|_{L^2}.
\end{align}
\end{enumerate}
\end{lem}

\begin{proof}
\noindent \emph{Step 1: Pointwise estimates and proof of \ref{item:infty}.}

\begin{align} 
    |\Id - w^\eps|(x - x_i^\eps) \lesssim \frac{\eps^\alpha}{|x - x_i^\eps|} \qquad \text{in } C_i^\eps \cup D_i^\eps, \label{pointwise.w^eps-1} \\
    |\nabla w^\eps| + |q^\eps| \lesssim\frac{\eps^\alpha}{|x - x_i^\eps|^2}  \qquad \text{in } C_i^\eps \cup D_i^\eps.  \label{pointwise.nabla.w^eps}
\end{align}
The estimates on $C_i^\eps$ follow immediately from standard decay estimates for the Stokes equations in exterior domains (cf. \cite[Theorem V.3.2]{Galdi11}) applied to $(w_k, q_k)$ from \eqref{w_k} and the definition of $w^\eps,q^\eps$ through  rescaling on $C_i^\eps$. 
Consequently, the estimates on $D_i^\eps$ are deduced from the estimates on $\partial D_i^\eps$ and standard regularity theory for the Stokes equations.

Clearly, \ref{item:infty}  follows directly from these pointwise estimates. 

\medskip

\noindent \emph{Step 2: Proof of \ref{item:L^2}.} 
Using \eqref{pointwise.w^eps-1} and $w^\eps = \Id$ in $K_i^\eps$, we compute for one cell, for all $p < 3$,
\begin{align} \label{L^p.single.cube}
    \|\Id-w^\eps\|^p_{L^p(Q_i^\eps)} \lesssim \eps^{\alpha p} \int_{B_{\eta_\eps/2}(x_i^\eps)} |x - x_i^\eps|^{-p} \dd x\lesssim \eta_\eps^{3 - p} \eps^{\alpha p} .
\end{align}
For any compact $K \subset \R^3$, we can cover $K$ by $C(K) \eps^{-3}$ many cubes $Q_i^\eps$. Hence,  $\|\Id - w^\eps\|^2_{L^2(K)} \lesssim C(K) (\eta_\eps/\eps) \eps^{2(\alpha -1)} \to 0 $ as $\eps \to 0$ since $\eta_\eps \leq \eps$ and $\alpha > 1$.

Denoting $(\varphi)_i = \fint_{Q_i^\eps} \varphi \dd x$, we have  for $p > 3/2$ by the Sobolev embedding $W^{2,p}(Q_i^\eps) \subset L^\infty(Q_i^\eps)$    and the Poincar\'e inequality that 
\begin{align}
     \|\varphi - (\varphi)_i - (\nabla \varphi)_i(x-x_i^\eps)\|_{L^\infty(Q_i^\eps)} \leq C_\eps \|\nabla^2 \varphi \|_{L^p(Q_i^\eps)}.
\end{align}
Scaling considerations imply $C_\eps = C \eps^{2 - 3/p}$. Thus, using also that $|(\psi)_i |\leq \eps^{-3/p} \|\psi\|_{L^p(Q_i^\eps)}$,
\begin{align}
   \|\varphi\|_{L^\infty(Q_i^\eps)} \leq   \|\varphi - (\varphi)_i - (\nabla \varphi)_i(x-x_i^\eps)\|_{L^\infty(Q_i^\eps)} +  |(\varphi)_i| + \eps |(\nabla \varphi)_i| \lesssim \eps^{- 3/p} \| \varphi \|_{W^{2,p}(Q_i^\eps)}.
\end{align}

Hence, for $p \in(3/2,3)$ 
\begin{align}
\begin{aligned}
    \|\varphi(\Id- w^\eps)\|^p_{L^p(\R^3)} &\lesssim \sum_i \eta_\eps^{3 - p} \eps^{\alpha p} \|\varphi\|^p_{L^\infty(Q_i^\eps)} 
    \lesssim \eta_\eps^{3 - p} \eps^{\alpha p-3} \|\varphi \|_{W^{2,p}(\R^3)}^p.
    \end{aligned}
    \label{1-w_eps.proof}
\end{align}
Estimates \eqref{w_eps-1.p=3}--\eqref{wk.qk.phi.1} are proved analogously. For \eqref{w_eps-1.p=3} we use in addition that $B_{\delta \eps^\alpha}(x^\eps_i) \subset \mathcal T_i^\eps $ for some $\delta >0$ that depends only on the reference particle $\mathcal T$. Therefore $w^\eps = 0$ in $B_{\delta \eps^\alpha}(x^\eps_i)$.

\medskip

\noindent \emph{Step 3: Proof of \ref{item:Hardy}:}
    It suffices to prove that for all $\varphi \in C^\infty(Q_i^\eps)$ with $\varphi = 0 $ in $\mathcal T_i^\eps$, we have
    \begin{align}
        \|  |\nabla w^\eps| \varphi^2\|_{L^1(Q_i^\eps)}  \lesssim \eta_\eps \|\nabla \varphi\|_{L^2(Q_i^\eps)}^2.
    \end{align}
        Without loss of generality, we assume $x_i^\eps = 0$.
    By the pointwise estimate \eqref{pointwise.nabla.w^eps} and the fundamental theorem of calculus, we have for all $x \in C_i^\eps \cup D_i^\eps$ with $\delta > 0$ as above 
    \begin{align}
        |\nabla w^\eps(x)| |\varphi(x)|^2 \lesssim \frac{\eps^\alpha}{|x|^2} |\varphi(x)|^2 \leq \frac{\eps^\alpha}{|x|^2} \left(\int_{\delta \eps^\alpha}^{|x|}  \Bigl|\nabla \varphi\Bigl(\frac{t x}{|x|}\Bigr)\Bigr|  \dd t \right)^2.
    \end{align}
    This implies
    \begin{align}
        \||\nabla w^\eps| \varphi^2\|_{L^1(Q_i^\eps)} &\lesssim \eps^\alpha \int_{S^2} \int_{\delta \eps^\alpha}^{\eta_\eps/2} |\varphi(r n)|^2 \dd r \dd n 
         \leq   \eps^\alpha \int_{S^2} \int_{\delta \eps^\alpha}^{\eta_\eps/2}  \left(\int_{\delta \eps^\alpha}^{\eta_\eps/2} |\nabla \varphi(t n)| \dd t \right)^2 \dd r \dd n \\
        &\lesssim \eta_\eps \eps^{\alpha} \int_{S^2} \int_{\delta \eps^\alpha}^{\eta_\eps/2} r^2 |\nabla \varphi(r n)|^2 \dd r \dd n \int_{\delta \eps^\alpha}^{\eta_\eps/2} \frac 1 {r^2} \dd r \lesssim  \eta_\eps \|\nabla \varphi\|_{L^2(Q_i^\eps)}^2,
    \end{align}
    as claimed. The proof of the estimate for the term involving $q^\eps$ is analogous.
\end{proof}

\begin{lem} \label{lem:M_eps} We can  write
\begin{align} \label{decomposition.M.gamma}
  - \Delta w^\eps +  \nabla q^\eps = \eps^{\alpha - 3} M_\eps - \gamma_\eps 
\end{align}
for some $M_\eps,\gamma_\eps \in W^{-1,\infty}(\R^3)$
where $\langle \gamma_\eps, v \rangle = 0$ for all $v \in H^1_0(\Omega_\eps)$ and, for all $\varphi \in  H^{3}(\R^3)$ and all $\psi \in H^1(\R^3)$,
\begin{align} \label{M.W^-1,infty}
    \langle (M_\eps - \mathcal R)\varphi, \psi \rangle  \lesssim \left(\eta_\eps^{-1} \eps^{\alpha } \|\psi\|_{L^2(\R^3)} + \eta_\eps^{-\frac 1 2} \eps^{\frac{3}2} \| \psi\|_{H^1(\R^3)} \right) \|\varphi\|_{H^3(\R^3)},
\end{align}
where the matrix $\mathcal R$ is defined in  \eqref{def:R}.
\end{lem}
\begin{proof}
    We observe that  
$-\Delta w^\eps + \nabla q^\eps$ is supported on $\bigcup_i \partial C_i^\eps \cup \partial D_i^\eps = \bigcup_i \partial D_i^\eps \cup \partial  \Omega_\eps$ and we define $\gamma_\eps$ to be the part supported on $\partial \Omega_\eps$ which consequently satisfies $\langle \gamma_\eps, v \rangle = 0$ for all $v \in H^1_0(\Omega_\eps)$. Then \eqref{decomposition.M.gamma} holds with $M_k^\eps$, the columns of $M^\eps$, being
    \begin{align} \label{M_eps^k}
        M^\eps_k  = \eps^{3 -\alpha} \sum_i \left( m^\eps_{k,i} + \dv(\1_{D_i^\eps} (q_k^\eps \Id - \nabla w_k^\eps))\right) 
    \end{align}
    where 
    \begin{align} \label{m_k^eps}
         m^\eps_{k,i}= \eps^{-\alpha} (q_k \Id - \nabla w_k)(\eps^{-\alpha} x) n |\partial B_{\eta_\eps/4}| \delta^i_{\eta_\eps/4}, && \delta^i_{\eta_\eps/4} = \frac{\mathcal H^2|_{\partial B_{\eta_\eps/4}(x_i^\eps)}}{|\partial B_{\eta_\eps/4}(x_i^\eps)|},
    \end{align}
    and where $w_k,q_k$ are as in \eqref{w_k} and  $n$ is the unit normal on $\partial B_{\eta_\eps/4}(x_i^\eps)$.
    By \cite[Lemma 2.3.5]{Allaire90a} (which follows from the fact that $w_k,q_k$ asymptotically behave as the fundamental solution of the Stokes equations), we have 
    \begin{align}
        m^\eps_{k,i} =  \frac {\eps^{\alpha}} 2 \left( \mathcal R_k + 3  (\mathcal R_k \cdot n) n + \eta_\eps^{-1}  \eps^{\alpha} r^\eps_{k,i}\right) \delta^i_{\eta_\eps/4}, &&
        \|r^\eps_{k,i}\|_{W^{1,\infty}(\partial B_{\eta_\eps/4})} \lesssim 1.
     \end{align}
     To conclude the proof, it suffices to show that for all $\varphi \in H^3(\R^3)$ and all $\psi \in H^1(\R^3)$
     \begin{align}
        \Bigl\|\varphi \Bigl(\mathcal R_k - \frac {\eps^{3}} 2 \sum_i  \left( \mathcal R_k + 3  (\mathcal R_k \cdot n) n\right)\delta^i_{\eta_\eps/4}\Bigr) \Bigr\|_{H^{-1}(\R^3)}  \lesssim \eta_\eps^{-\frac 1 2} \eps^{\frac{3} 2 } \|\varphi\|_{H^3(\R^3)},  \label{est:M.main} \\
          \eps^{3-\alpha} \Bigl\|\varphi \sum_i \dv(\1_{D_i^\eps} (q_k^\eps \Id - \nabla w_k^\eps)) \Bigr\|_{H^{-1}(\R^3)}  \lesssim \eta_\eps^{-\frac 1 2} \eps^{\frac{3} 2 } \|\varphi\|_{H^3(\R^3)} \label{est:M.remainder.2}, \\
         \Bigl\langle \varphi \eps^{3} \sum_i r^\eps_{k,i} \delta^i_{\eta_\eps/4} , \psi \Bigr\rangle  \lesssim \|\varphi\|_{H^3(\R^3)}\Bigl(\|\psi\|_{L^2(\R^3)} +  \eta_\eps^{-\frac 1 2} \eps^{\frac{3} 2 } \| \psi\|_{H^1(\R^3)} \Bigr). \qquad \label{est:M.remainder.1} 
     \end{align}
     Indeed, $\eta_\eps^{-1} \eps^\alpha \leq 1$ by assumption and thus \eqref{M_eps^k}--\eqref{est:M.remainder.1} imply the assertion.
     
     To prove \eqref{est:M.main}, we begin by observing that for all $v \in  H^1(Q_i^\eps)$ we have due to Sobolev embedding
     \begin{align} \label{Sobolev-Poincare}
         \|v - (v)_i\|_{L^6({Q_i^\eps})} \leq C \|\nabla v\|_{L^2({Q_i^\eps})},
     \end{align}
     where we recall the notation $(v)_i = \fint_{Q_i^\eps} v$ and where the constant $C$ is universal due to scaling considerations. Similarly, we have the Poincaré-type inequality
         \begin{align} \label{Poincare.boundary}
         \fint_{\partial B_{\eta_\eps /4}(x_i^\eps)} \Bigl| v - \fint_{B_{\eta_\eps /4}(x_i^\eps)} v  \dd x \Bigr| \dd y \lesssim \eta_\eps^{-\frac 1 2} \|\nabla v\|_{L^2(B_{\eta_\eps /4}(x_i^\eps))}.
     \end{align}
     Since 
     \begin{align}
         \fint_{\partial B_{\eta_\eps/4}} \frac 1 2 \left( \mathcal R_k + 3  (\mathcal R_k \cdot n) n\right) \dd x = \mathcal R_k,
     \end{align}
     we deduce that for any $v \in H^1(Q_i^\eps)$ that
     \begin{align}
        & \Bigl| \int_{Q_i^\eps} \Bigl(v \cdot \mathcal R_k -  \frac 1 2 \fint_{\partial B_{\eta_\eps/4}(x_i^\eps)} v \cdot \left( \mathcal R_k + 3  (\mathcal R_k \cdot n) n\right) \dd y \Bigr) \dd x \Bigr|  \\
         & = \frac {\eps^3 } 2   \Bigl| \fint_{\partial B_{\eta_\eps/4}(x_i^\eps)} (v - (v)_i)  \cdot \left( \mathcal R_k + 3  (\mathcal R_k \cdot n) n \right) \dd x \Bigr| \\
         &\lesssim \eta_\eps^{-\frac 1 2} \eps^{3}  \|\nabla v\|_{L^2( B_{\eta_\eps/4}(x_i^\eps))} + \eps^3  \fint_{B_{\eta_\eps/4}(x_i^\eps)} |v - (v)_i| 
        \dd x\\
         &\lesssim \eta_\eps^{-\frac 1 2} \eps^{3}  \|\nabla v\|_{L^2( B_{\eta_\eps/4}(x_i^\eps))} + \eta_\eps^{-\frac 1 2} \eps^{3} \||v - (v)_i\|_{L^6(B_{\eta_\eps/4}(x_i^\eps))} \\
         & \lesssim \eta_\eps^{-\frac 1 2} \eps^{3}  \|\nabla v\|_{L^2(Q_i^\eps)}.
     \end{align}
     Therefore, for $\varphi \in H^3(\R^3)$ and $\psi \in H^1(\R^3)$,
     \begin{align}
          \Bigl\langle \varphi \Bigl(\mathcal R_k -  \frac {\eps^{3}} 2 \sum_i  \left( \mathcal R_k + 3  (\mathcal R_k \cdot n) n\right)\delta^i_{\eta_\eps/4} \Bigr), \psi  \Bigr\rangle 
          & \lesssim  \eta_\eps^{-\frac 1 2}\eps^{3}  \sum_i \|\nabla (\varphi \psi)\|_{L^2(Q_i^\eps)} \\
          &\lesssim \eta_\eps^{-\frac 1 2} \eps^{\frac{3} 2 } \|\psi\|_{H^1(\R^3)} \Bigl( \eps^3 \sum_i \|\varphi\|_{W^{1,\infty}(Q_i^\eps)}^2 \Bigr)^{\frac 1 2} \\
          & \lesssim  \eta_\eps^{-\frac 1 2}\eps^{\frac{3} 2 } \|\psi\|_{H^1(\R^3)} \|\varphi\|_{H^3(\R^3)},
     \end{align}
     where the last inequality is shown as in \eqref{1-w_eps.proof}
     
     \smallskip
     
          We turn to \eqref{est:M.remainder.2}. We use the pointwise estimates \eqref{pointwise.nabla.w^eps} to bound
     \begin{align}
         \eps^{3-\alpha} \Bigl\langle \varphi \sum_i \dv(\1_{D_i^\eps} (q_k^\eps \Id - \nabla w_k^\eps)),\psi \Bigr\rangle &\lesssim \eps^{3-\alpha} \eta_\eps^{\frac 3 2}  \sum_i \|q_k^\eps \Id - \nabla w_k^\eps\|_{L^\infty(D_i^\eps)}  \|\psi\|_{H^1(Q_i^\eps)} \|\varphi\|_{W^{1,\infty}(Q_i^\eps)} \\
         & \lesssim \eps^{3 - \alpha}  \eta_\eps^{\frac 3 2}  \eta_\eps^{-2}   \eps^{\alpha} \eps^{-3/2} \|\psi\|_{H^1(\R^3)} \|\varphi\|_{H^3(\R^3)} \\
         &= \eta_\eps^{-\frac 1 2} \eps^{\frac{3}2}\|\psi\|_{H^1(\R^3)} \|\varphi\|_{H^3(\R^3)} .
     \end{align}

     It remains to show \eqref{est:M.remainder.1}. Using again \eqref{Sobolev-Poincare} and  \eqref{Poincare.boundary},
     we have for any $v \in H^1(Q_i^\eps)$
     \begin{align}
         \Bigl| \fint_{\partial B_{\eta_\eps/4}(x_i^\eps)} v \dd x \Bigr|& \lesssim   \fint_{\partial B_{\eta_\eps/4}(x_i^\eps)} \Bigl|v  - \fint_{ B_{\eta_\eps/4}(x_i^\eps)} v \dd y \Bigr| \dd x + \fint_{ B_{\eta_\eps/4}(x_i^\eps)} |v  - (v)_i| \dd x   + |(v)_i|   \\
         &\lesssim   \eta_\eps^{-\frac 1 2}   \|\nabla v\|_{L^2(Q_i^\eps)} + \eps^{-3/2} \|v\|_{L^2(Q_i^\eps)}.
     \end{align}
     Thus, for  $\varphi \in H^3(\R^3)$ and $\psi \in H^1(\R^3)$, using \eqref{m_k^eps},
     \begin{align}
         \Bigl\langle \varphi \eps^{3} \sum_i r^\eps_{k,i} \delta^i_{\eta_\eps/4}, \psi \Bigr\rangle &\lesssim  \eta_\eps^{-\frac 1 2} \eps^{3 } \sum_i \|\psi\|_{H^1(Q_i^\eps)} \|\varphi\|_{W^{1,\infty}(Q_i^\eps)} +  \eps^{3/2} \sum_i \|\psi\|_{L^2(Q_i^\eps)} \|\varphi\|_{L^{\infty}(Q_i^\eps)}  \\
         & \lesssim  \|\varphi\|_{H^3(\R^3)}\left(  \eta_\eps^{-\frac 1 2} \eps^{\frac{3} 2 } \| \psi\|_{H^1(\R^3)} + \|\psi\|_{L^2(\R^3)}  \right).
     \end{align}
     This finishes the proof.
\end{proof}

\medskip

\begin{lem} \label{lem:Bogovski}
      For all $1 < p < \infty$, there exists a linear operator $\mathcal B_\eps: W^{1,p}(\R^3) \to W^{1,p}_0(\Omega_\eps)$ such that for all $\varphi \in  W^{1,p}(\R^3)$ that are divergence free we have
      \begin{align} \label{Bog}
          \dv \mathcal B_\eps(\varphi) = w^\eps : \nabla \varphi
      \end{align}
      and
      \begin{align} \label{Bog.H^1}
          \|\nabla \mathcal B_\eps(\varphi) \|_{L^p} \lesssim \|(\Id - w^\eps) : \nabla \varphi\|_{L^p},  && 
          \|\mathcal B_\eps(\varphi) \|_{L^p} \lesssim \eta_\eps\|(\Id - w^\eps) : \nabla \varphi\|_{L^p}.
      \end{align}
\end{lem}
\begin{proof}
    It suffices to construct the linear operator on the subspace of divergence free functions $\varphi \in W^{1,p}(\R^3)$. We observe that then $w^\eps : \nabla \varphi = 0$ in $\R^3 \setminus A_i^\eps$ where $A_i^\eps := C_i^\eps \cup D_i^\eps$ and, since the functions $w_k^\eps$ are divergence free,
    \begin{align}
        \int_{A_i^\eps} w^\eps : \nabla \varphi \dd x = \int_{A_i^\eps \cup \mathcal T_i^\eps} w^\eps : \nabla \varphi \dd x = \int_{A_i^\eps \cup \mathcal T_i^\eps} \dv((w^\eps-\Id) \varphi) \dd x =  0
    \end{align}
    as $w^\eps = \Id$ on $\partial D_i^\eps$. Therefore we may employ a Bogovski operator in $A_i^\eps$. More precisely, by \cite[Lemma 3.1]{DieningFeireislLu17} (which is a consequence of \cite{AcostaDuranMuschietti06} and \cite{DieningRuzickaSchumacher10}), there exist operators $\mathcal B^\eps_i : L^p_0(A_i^\eps) \to W^{1,p}_0(A_i^\eps)$ ($L^p_0$ denotes the subspace of $L^p$ functions with vanishing mean) such that for all $h \in L_0^p(A_i^\eps)$
    \begin{align}
        \dv \mathcal B^\eps_i (h) = h, && \|\mathcal B^\eps_i (h) \|_{W^{1,p}_0(A_i^\eps)} \lesssim \|h\|_{L^p_0(A_i^\eps)}.
    \end{align}
    We then deduce that $\mathcal B_\eps(\varphi) := \sum_i \mathcal B^\eps_i( w^\eps : \nabla \varphi)$ satisfies \eqref{Bog} as well as the first inequality in \eqref{Bog.H^1}.
    The second inequality in  \eqref{Bog.H^1} follows from the first one and the Poincar\'e inequality in the domains $A_i^\eps \subset B_{\eta_\eps}(x_i)$.
\end{proof}

For the treatment of the supercritical case, we will rely on the following Poincaré inequality in $\Omega_\eps$. 
It is proved in \cite[Lemma 3.4.1]{Allaire90b} when $\Omega_\eps$ is a bounded domain. Since the proof is based on a local Poincaré inequality in each of the cubes $Q_i^\eps$, it still applies here.
\begin{prop}[{\cite[Lemma 3.4.1]{Allaire90b}}] \label{prop:Poincare}
    For all $\varphi \in H^1_0(\Omega_\eps)$
    \begin{align} \label{Poincare.Omega_eps}
        \|\varphi\|_{L^2(\Omega_\eps)} \lesssim \eps^{\frac{3-\alpha} 2 } \|\nabla \varphi \|_{L^2(\Omega_\eps)}.
    \end{align}
\end{prop}

\section{Proof of the main results} \label{sec:proof.main}

As outlined in Section \ref{sec:strategy}, the strategy for the proof of the main results is based on energy estimates for the difference
\begin{align} \label{v_eps}
    v_\eps = w^\eps u - u_\eps - \mathcal B_\eps(u).
\end{align}
Here $u_\eps$ is the solution to \eqref{Navier.Stokes.holes} in the  critical and subcritical case and to \eqref{Navier.Stokes.holes.rescaled} in the supercritical case and $u$ is the solution to \eqref{Euler-Darcy}, \eqref{Euler} and \eqref{Darcy}, respectively.
Moreover, $w^\eps$ is the matrix valued function defined at the beginning of Section \ref{sec:w} and depends on a parameter $\eps^\alpha \leq \eta_\eps \leq \eps$ that we will choose later.
Finally, $\mathcal B_\eps$ is the operator from Lemma \ref{lem:Bogovski}.

We first observe that the difference $(w^\eps - \Id) u - \mathcal B_\eps(u)$ between $v_\eps$ and $u -  u_\eps$ is very small, namely
     \begin{align}
    \label{v_eps.to.u_eps.2}
    \|v_\eps - (u -  u_\eps) \|_{L^\infty(0,T;L^2(\R^3))}     &\leq C \eta_\eps^{\frac 1 2} \eps^{\alpha -  \frac{ 3}{2}},
\end{align}
where the constant $C$ depends only on  $\mathcal T$ and $\|u\|_{L^\infty(0,T;H^3(\R^3))}$.
Indeed, this follows immediately from \eqref{w_eps-1} and \eqref{Bog.H^1}.



\subsection{Proof of Theorem \ref{th:Euler-Brinkman} and Theorem \ref{th:Euler}} \label{sec:proof.subcritical}

Throughout this subsection, we assume that the parameters $\alpha$ and $\gamma$ are in the range of the critical or subcritical regime specified in Theorem \ref{th:Euler-Brinkman} and \ref{th:Euler}, respectively, that is
$\alpha > 3/2$ and $\gamma >0$, $\gamma \in [3-\alpha,\alpha)$ or $\gamma= \alpha$ and $\mu_0 \gg 1$.
Moreover, $v_\eps$ is defined by \eqref{v_eps} where $u_\eps$ is the solution to \eqref{Navier.Stokes.holes} and $u$ is the solution to \eqref{Euler} or \eqref{Euler-Darcy}.


The main technical part of the proof of the main results is an energy estimate for $v_\eps$ stated in the following proposition. Thereafter, we show how Theorem \ref{th:Euler-Brinkman} and Theorem  \ref{th:Euler} follow from this proposition and Gronwall's inequality. 

\begin{prop} \label{pro:v_eps}
Let  $\eps^\alpha \leq \eta_\eps \leq \eps$. Then,
\begin{enumerate}[(i)]
    \item Then, under the assumptions of Theorem \ref{th:Euler-Brinkman} we have for all $t \leq T$
  \begin{align} 
  \label{est.v_eps}
  \begin{aligned}
        \|v_\eps(t)\|_{L^2(\Omega_\eps)}^2 + &(\eps^\gamma - C \eta_\eps) \|\nabla v_\eps\|^2_{L^2((0,t) \times \Omega_\eps)} 
       \\
       &\leq \|v_\eps(0)\|_{L^2(\Omega_\eps)}^2 + C\|(f_\eps - f)\|_{L^2((0,T) \times \Omega_\eps)}^2  
    +     C \|v_\eps\|^2_{L^2((0,t) \times \Omega_\eps)} \\
   &+C \left(\eta_\eps \eps^{2 \alpha- \gamma-3}  + \eta_\eps^{-1} \eps^{3- \gamma}  + \eps^{2\gamma} +  \eta_\eps^2 \right) \quad 
   \end{aligned}
  \end{align}
  for some constant $C$ which depends only on  $\mathcal T$, $T$ $\|f\|_{L^\infty(0,T;H^2(\R^3))}$, $\|u\|_{C^1(0,T;H^4(\R^3))}$ and $\|\nabla p\|_{L^\infty(0,T;H^2(\R^3))}$.
  
  \item   Under the assumptions of Theorem \ref{th:Euler} we have for all $t \leq T$
  \begin{align} \label{est.v_eps.Euler}
  \begin{aligned}
        \|v_\eps(t)\|_{L^2(\Omega_\eps)}^2 + &(\mu_0 \eps^\gamma - C \eta_\eps) \|\nabla v_\eps\|^2_{L^2((0,t)\times\Omega_\eps)} \\
        &\leq \|v_\eps(0)\|_{L^2(\Omega_\eps)}^2 + C\|(f_\eps - f)\|_{L^2(0,T)\times \Omega_\eps)}^2 
        +    C \|v_\eps\|^2_{L^2(0,t;L^2(\Omega_\eps))}  \\ &+C_{\mu_0}\left(\eps^{2 \alpha  + 2 \gamma-6} + \eta_\eps \eps^{2 \alpha- \gamma-3} + \eta_\eps^{-1} \eps^{2 \alpha + \gamma-3}   + \eps^{2\gamma} +  \eta_\eps^2 \right)
        \end{aligned}
  \end{align}
  for some $C$ which depends only on  $\mathcal T$, $T$, $\|f\|_{L^\infty(0,T;H^2(\R^3))}$, $\|u\|_{C^1(0,T;H^4(\R^3))}$, $\|\nabla p\|_{L^\infty(0,T;H^2(\R^3))}$ and some $C_{\mu_0}$ which depends additionally on $\mu_0$.
  \end{enumerate}
\end{prop}

 \begin{proof}[Proof of Theorem \ref{th:Euler-Brinkman}]
 We choose $\eta_\eps = \frac 1 C \eps^\beta $ with $\beta = \max\{1,\gamma\}$ such that we may drop the second term on the left-hand side of \eqref{est.v_eps}. Note that as $\gamma = 3 -\alpha$ and $\alpha \in (3/2,3)$, the assumption $\eps^\alpha \leq \eta_\eps \leq \eps$ is satisfied for all $\eps$ sufficiently small (for $\eps$ of order $1$, the assertion of the theorem is an immediate consequence of the energy inequality \eqref{energy.inequality}).
 
 Then, by Gronwall's inequality, Proposition \ref{pro:v_eps} yields
  \begin{align}
      \|v_\eps(t)\|_{L^2(\Omega_\eps)}^2 \lesssim \|v_\eps(0)\|_{L^2(\Omega_\eps)}^2 + \|(f_\eps - f)\|_{L^2(0,T;L^2(\Omega_\eps))}^2 + \left(\eps^{2 \alpha -3}  +  \eps^{6 - 2\alpha} \right) 
  \end{align}
  and we deduce with \eqref{v_eps.to.u_eps.2}, which only gives a higher order error, that
  \begin{align}
      \|(u_\eps - u)(t)\|_{L^2(\Omega_\eps)}^2 \lesssim \|(u_\eps - u)(0)\|_{L^2(\Omega_\eps)}^2 +\|(f_\eps - f)\|_{L^2(0,T;L^2(\Omega_\eps))}^2+  \left(\eps^{2 \alpha -3}  +  \eps^{6 - 2\alpha} \right) .
  \end{align}
  This finishes the proof.
 \end{proof}
 
\begin{proof}[Proof of Theorem \ref{th:Euler}]
    We choose $\eta_\eps = \delta \eps^\beta$ with $\beta = \max\{\gamma,1\}$
    and 
    \begin{align}
        \delta = \begin{cases} 1 &\qquad \text{if } \gamma = \alpha, \\
        \frac {\mu_0} C & \qquad \text{if } \gamma < \alpha.
        \end{cases}
    \end{align}
    This choice guarantees that $\eps^\alpha \leq \eta_\eps \leq \eps$ is satisfied for all $\eps$ sufficiently small.
    Moreover, choosing $M = C$, the assumption $\mu_0 \geq M$ if $\gamma = \alpha$ allows us to drop the second term on the left-hand side in \eqref{est.v_eps.Euler} in all cases. Therefore, arguing as in the proof above yields
    \begin{align}
      \|(u_\eps - u)(t)\|_{L^2(\Omega_\eps)}^2 &\lesssim \|(u_\eps - u)(0)\|_{L^2(\Omega_\eps)}^2 +\|(f_\eps - f)\|_{L^2(0,T;L^2(\Omega_\eps))}^2  \\
      &  +  \left(\eps^{2 \alpha + 2\gamma-6} + \eps^{2 \alpha-3} + \eps^{2 \alpha + \gamma- 4}   + \eps^{2\gamma} \right).
  \end{align}
   We observe that $2 \alpha  + \gamma - 4 \geq \min\{2\alpha -3, 2 \alpha + 2\gamma - 6\}$ to finish the proof.
\end{proof}

\begin{proof}[Proof of Proposition \ref{pro:v_eps}]
We focus on the critical case $\gamma = 3 -\alpha$ where $u$ solves \eqref{Euler-Darcy}. We discuss the necessary adaptions for the subcritical case $\gamma > 3- \alpha$ in the last step of the proof.
Throughout the proof we write $\lesssim$ for $\leq C$ with $C$ as specified in the statement of the proposition.

\medskip

\noindent \emph{Step 1: PDE solved by $\check u_\eps := w^\eps u - \mathcal B_\eps(u)$:}
We observe that $\check u_\eps$ satisfies $\check u_\eps = 0$ on $(0,T) \times \partial  \Omega_\eps$ and, in $(0,T) \times \Omega_\eps$
\begin{align} \label{PDE.check.u_eps}
\begin{aligned}
    \partial_t \check u_\eps - \eps^\gamma \Delta \check u_\eps  + w^\eps (u \cdot \nabla u) + w^\eps \nabla p &= w^\eps f +  (M_\eps  -  w^\eps \mathcal R) u - \eps^\gamma \nabla q^\eps u \\
    &- 2 \eps^\gamma \nabla w^\eps \nabla u - \eps^\gamma w^\eps \Delta u + \mathcal B_\eps(\partial_t u) + \eps^\gamma \Delta \mathcal B_\eps(u),
    \end{aligned}
 \end{align}
with $M_\eps$ as in \eqref{decomposition.M.gamma}. Moreover, $\dv \check u_\eps = 0$.

\emph{Step 2: Relative energy inequality:}
We consider the relative energy $\frac 1 2 \|v_\eps\|_{L^2}^2$.
We estimate using the energy inequality \eqref{energy.inequality} for $u_\eps$ as well as $\check u_\eps \in L^\infty(H^1)$, $\partial_t \check u_\eps \in L^1(H^{-1})$
\begin{align} \label{energy.0}
    \frac 1 2 \|v_\eps(t)\|_{L^2(\Omega_\eps)}^2 &= \frac 1 2 \|u_\eps(t)\|_{L^2(\Omega_\eps)}^2  -  ( \check u_\eps(t),u_\eps(t))_{L^2(\Omega_\eps)}  + \frac 1 2 \|\check u_\eps(t)\|_{L^2(\Omega_\eps)}^2 \\
    & \leq \frac 1 2 \|v_\eps(0)\|_{L^2(\Omega_\eps)}^2 - \eps^\gamma \int_0^t \|\nabla u_\eps\|_{L^2(\Omega_\eps)}^2 \dd s + \int_0^t \int_{\Omega_\eps} f_\eps \cdot u_\eps \dd x \dd s \\
    &- \int_0^t  \int_{\Omega_\eps} \left(\partial_t \check u_\eps \cdot u_\eps + \partial_t u_\eps \cdot \check u_\eps\right) \dd x \dd s + \int_0^t \int_{\Omega_\eps} \partial_t \check u_\eps \cdot \check u_\eps \dd x \dd s.
\end{align}
Using the equation solved by $u_\eps$, we have
\begin{align} \label{energy.1}
   - \int_0^t \int_{\Omega_\eps} \partial_t u_\eps \cdot \check u_\eps \dd x \dd s =  \int_0^t \int_{\Omega_\eps} \left( (u_\eps \cdot \nabla u_\eps) \cdot \check u_\eps  + \eps^\gamma  \nabla u_\eps \cdot \nabla \check u_\eps- f_\eps \cdot \check u_\eps \right)\dd x \dd s
\end{align}
and likewise, using the equation of $\check u_\eps$
\begin{align} \label{energy.2}
    \int_0^t \int_{\Omega_\eps}  \partial_t \check u_\eps \cdot v_\eps &=  -  \int_0^t  \int_{\Omega_\eps} \left(\eps^\gamma \nabla \check u_\eps \cdot \nabla v_\eps  +  (w^\eps (u \cdot \nabla u) ) \cdot (\check u_\eps - u_\eps) -  (w^\eps f +  \tilde F_\eps) \cdot v_\eps \right) \dd x \dd s \qquad 
\end{align}
where 
\begin{align}
    \tilde F_\eps = - w^\eps \nabla p +  (M_\eps  - w^\eps \mathcal R ) u - \eps^\gamma \nabla q^\eps u - 2 \eps^\gamma\nabla w^\eps \nabla u -\eps^\gamma w^\eps \Delta u
    + \mathcal B_\eps(\partial_t u) + \eps^\gamma \Delta \mathcal B_\eps(u).
\end{align}
inserting \eqref{energy.1}--\eqref{energy.2} in \eqref{energy.0} and denoting
\begin{align}
    F_\eps = \tilde F_\eps + (w^\eps f -f_\eps) 
\end{align}
yields 
\begin{align} \label{energy.v_eps}
     \frac 1 2 \|v_\eps\|_{L^2}^2(t) +\eps^\gamma  \int_0^t \|\nabla v_\eps\|_{L^2(\Omega_\eps)}^2 \dd s &\leq   \int_0^t \int_{\Omega_\eps} \Bigl( (u_\eps \cdot \nabla u_\eps) \cdot \check u_\eps  -  (w^\eps (u \cdot \nabla u) ) \cdot v_\eps  +  F_\eps \cdot v_\eps \Bigr) \dd x \dd s. \quad 
\end{align}
Thus, we deduce
\begin{align} \label{PDE.v_eps}
    \frac 1 2 \|v_\eps(t)\|_{L^2(\Omega_\eps)}^2  + \eps^\gamma \|\nabla v_\eps\|_{L^2((0,t) \times \Omega_\eps)}^2 &\leq \frac 1 2 \|v_\eps(0)\|_{L^2(\Omega_\eps)}^2 +  |I_1| + |I_2|
\end{align}
where
\begin{align}
    I_1 &= \int_0^t \int_{\Omega_\eps} \Bigl( (u_\eps \cdot \nabla u_\eps) \cdot \check u_\eps  -  (w^\eps (u \cdot \nabla u) ) \cdot v_\eps \Bigr) \dd x \dd s, \\
    I_2 &= \int_0^t \int_{\Omega_\eps}  F_\eps \cdot v_\eps \dd x \dd s .
\end{align}

    \noindent \emph{Step 3: Bound of $I_1$:}  
    We first manipulate the first term in $I_1$. Using $u_\eps = \check u_\eps = 0$ on $\partial \Omega_\eps$ as well as $\dv u = \dv u_\eps = 0$ yields by integration by parts
    \begin{align}\label{manipulate.trilinear}
    \begin{aligned}
        \int_0^t \int_{\Omega_\eps} (u_\eps \cdot \nabla u_\eps) \cdot \check u_\eps \dd x \dd s &= -\int_0^t \int_{\Omega_\eps}  (u_\eps \cdot \nabla \check u_\eps) \cdot ( u_\eps -\check u_\eps)  \dd x \dd s \\
        &= -\int_0^t \int_{\Omega_\eps}  (v_\eps \cdot \nabla \check u_\eps) \cdot v_\eps   \dd x \dd s + \int_0^t \int_{\Omega_\eps}  (\check u_\eps \cdot \nabla \check u_\eps) \cdot v_\eps   \dd x \dd s.
    \end{aligned}
    \end{align}
    This allows us to rewrite
    \begin{align}
        I_1 &= -\int_0^t \int_{\Omega_\eps}  (v_\eps \cdot \nabla \check u_\eps) \cdot v_\eps   \dd x \dd s + \int_0^t \int_{\Omega_\eps} (\Id - w^\eps) ( u \cdot \nabla u) \cdot v_\eps   \dd x \dd s \\
        &+ \int_0^t \int_{\Omega_\eps}  ( ( \check u_\eps - u) \cdot \nabla  u) \cdot v_\eps   \dd x \dd s + \int_0^t \int_{\Omega_\eps}  ( \check u_\eps \cdot \nabla  (\check u_\eps - u)) \cdot v_\eps   \dd x \dd s =: I_1^1 + I_1^2 + I_1^3 + I_1^4.
    \end{align}
    We recall $\check u_\eps = w^\eps u - \mathcal B_\eps(u)$ to estimate by the regularity assumptions of $u$, \eqref{Hardy} and \eqref{Bog.H^1} combined with \eqref{w_eps-1.p=3} and another integration by parts
    \begin{align} \label{I_1^1}
    \begin{aligned}
        |I_1^1| &\lesssim \|v_\eps\|_{L^2(0,t;L^2(\Omega_\eps))}^2 \|w_\eps\|_{L^\infty(\R^3)} \|\nabla u\|_{L^\infty(0,t;L^\infty(\R^3))} + \|\nabla w_\eps |v_\eps|^2\|_{L^1(0,t;L^1(\Omega_\eps))} \|u\|_{L^\infty(0,t;L^\infty(\R^3))} \\
        &+  \|\nabla v_\eps\|_{L^2(0,t;L^2(\Omega_\eps))} \| v_\eps\|_{L^2(0,t;L^6(\Omega_\eps))} \|\mathcal B_\eps(u)\|_{L^\infty(0,T;L^3(\Omega_\eps)} \\
        & \lesssim \|v_\eps\|_{L^2(0,t;L^2(\Omega_\eps))}^2 + \eta_\eps \left(1 + \eps^{\alpha -1} |\log \eps|^{\frac 13} \right) \|\nabla v_\eps\|_{L^2(0,t;L^2(\Omega_\eps))}^2 \\
        & \lesssim \|v_\eps\|_{L^2(0,t;L^2(\Omega_\eps))}^2 + \eta_\eps   \|\nabla v_\eps\|_{L^2(0,t;L^2(\Omega_\eps))}^2,
    \end{aligned}
    \end{align}
    where we used $\alpha  > 1$ in the last estimate.
    
    By the regularity assumptions of $u$ and \eqref{w_eps-1}, we have 
    \begin{align}
        |I_1^2| \lesssim   \|v_\eps\|_{L^2(0,t;L^2(\Omega_\eps))}^2 + \eta_\eps  \eps^{2 \alpha  - 3}.
    \end{align}
    Similarly, relying additionally on \eqref{Bog.H^1},
    \begin{align}
        |I_1^3| \lesssim \|v_\eps\|_{L^2(0,t;L^2(\Omega_\eps))}^2 + \eta_\eps  \eps^{2 \alpha - 3}.
    \end{align}
    Finally,   we estimate by another integration by parts
    \begin{align}
        |I_1^4| &\leq \frac 1 4 \eps^\gamma \|\nabla v_\eps\|_{L^2(0,t;L^2(\Omega_\eps))}^2 +  \eps^{-\gamma} \|\check u_\eps| \check u_\eps - u|\|_{L^2(0,t;L^2(\Omega_\eps))}^2.
    \end{align}
    We estimate using that $u$ and $w^\eps$ are uniformly bounded in $L^\infty$ as well as \eqref{Bog.H^1}, \eqref{w_eps-1} and Sobolev embedding
    \begin{align}
        \|\check u_\eps| \check u_\eps - u|\|_{L^2(0,t;L^2(\Omega_\eps))}^2 &\lesssim \int_0^t \left(\|(w^\eps - \Id) u\|_{L^2(\Omega_\eps))}^2 +  \|\mathcal B_\eps(u)\|_{L^2(\Omega_\eps)}^2 + \|\mathcal B_\eps(u)\|_{L^4(\Omega_\eps))}^2 \right) \dd s \\
        &\lesssim \int_0^t \left(\|(w^\eps - \Id) u\|_{L^2(\Omega_\eps))}^2 +  \|\mathcal B_\eps(u)\|_{L^2(\Omega_\eps)}^2 + \|\nabla \mathcal B_\eps(u)\|_{L^2(\Omega_\eps))}^2 \right) \dd s \\
        & \lesssim \eta_\eps  \eps^{2\alpha - 3}.
    \end{align}
    In summary,  we find,
    \begin{align}  \label{I_1}
        |I_1| \leq C \|v_\eps\|_{L^2(0,t;L^2(\Omega_\eps))}^2 + \left(\frac 1 4 \eps^\gamma + C \eta_\eps  \right)  \|\nabla v_\eps\|_{L^2(0,t;L^2(\Omega_\eps))}^2 + C \eta_\eps  \eps^{2\alpha  - \gamma -3}.
    \end{align}
    \medskip
    
    \noindent \emph{Step 4: Bound of $I_2$:}
     We  split 
     \begin{align}
         I_2 = I_2^1 + I_2^2 + I_2^3 + I_2^4 
     \end{align}
     where
     \begin{align}
         I_2^1 &= \int_0^t \int_{\Omega_\eps} ((\Id - w^\eps) (\nabla p - f) + f - f_\eps) \cdot  v_\eps \dd x \dd s, \\
         I_2^2 &= \int_0^t \int_{\Omega_\eps} ((M_\eps -  w^\eps \mathcal R) u )\cdot v_\eps \dd x \dd s, \\
         I_2^3 &= - \eps^\gamma \int_0^t \int_{\Omega_\eps} ( 2  \nabla w^\eps \nabla u  + w^\eps \Delta u  + \nabla q^\eps u) \cdot v_\eps \dd x \dd s, \\
         I_2^4 &= \int_0^t \int_{\Omega_\eps} \left( \mathcal B_\eps(\partial_t u)  \cdot v_\eps + \eps^\gamma \nabla \mathcal B_\eps( u) \nabla v_\eps \right) \dd x \dd s.
     \end{align}
     We estimate
     \begin{align*}
          |I_2^1| &\lesssim  \|(w^\eps -\Id) \nabla p\|^2_{L^2((0,t)\times \Omega_\eps)} + \|(w^\eps -\Id) f\|^2_{L^2((0,t)\times \Omega_\eps)} + \|f_\eps -f\|^2_{L^2((0,t)\times \Omega_\eps)}  +   \|v_\eps\|_{L^2((0,t)\times \Omega_\eps)}^2 \\
          & \lesssim \eta_\eps  \eps^{2\alpha - 3} + \|f_\eps -f\|^2_{L^2((0,t)\times \Omega_\eps)}  + \|v_\eps\|_{L^2((0,t)\times \Omega_\eps)}^2.
     \end{align*}
    We  rewrite
    \begin{align}
        I_2^2  =   \int_0^t \int_{\Omega_\eps}  (w^\eps -\Id) \mathcal Ru  \cdot v_\eps \dd x \dd t  + \int_0^t \langle  (M_\eps - \mathcal R)u, v_\eps \rangle \dd s. 
    \end{align}
    The first term on the right-hand side is estimated as above.
    Combining this with \eqref{M.W^-1,infty} to estimate the second term on the right-hand side yields  for some $\delta > 0$ to be chosen later
    \begin{align}
        | I_2^2| &\leq C \eta_\eps  \eps^{2\alpha  - 3} + C \eta_\eps^{-2} \eps^{2\alpha} +  \|v_\eps\|_{L^2(0,t;L^2(\Omega_\eps))}^2  + C_\delta \eta_\eps^{-1}  \eps^{3} C \eps^{-\gamma} + \delta  \eps^{\gamma} \|\nabla v_\eps\|^2_{L^2(0,t;L^2(\Omega_\eps))}  \\
        &\leq C \eta_\eps \eps^{2\alpha - 3} + C_\delta \eta_\eps^{-1}  \eps^{3 -\gamma} +   \|v_\eps\|_{L^2(0,t;L^2(\Omega_\eps))}^2  + \delta  \eps^{\gamma}  \|\nabla v_\eps\|^2_{L^2(0,t;L^2(\Omega_\eps))}.
    \end{align}
    where we used that $\eta_\eps \geq \eps^\alpha$ and $\alpha \geq 3 - \gamma$ to absorb the term $\eta_\eps^{-2} \eps^{2\alpha}$.
     Next, we estimate using \eqref{wk.qk.phi.1} and \eqref{Hardy}
     \begin{align}
         |I_2^3| &\leq C \eps^{\gamma} \int_0^t  \Bigl(\|(|\nabla w^\eps|^{\frac 1 2 } + |q_k^\eps|^{\frac 1 2 }) \nabla u\|_{L^2} \|(|\nabla w^\eps|^{\frac 1 2 } + |q_k^\eps|^{\frac 1 2 })  v_\eps\|_{L^2}  +  \|w^\eps\|_\infty \|v_\eps\|_{L^2} \Bigr)  \dd s \\
         &\lesssim C_\delta  \eps^{\gamma}\eta_\eps^2 \eps^{\alpha -3} + \delta\eps^{\gamma} \|\nabla v_\eps\|_{L^2(0,t;L^2(\Omega_\eps))}^2 + C \eps^{2\gamma} +   \|v_\eps\|_{L^2(0,t;L^2(\Omega_\eps)}^2 \\
         & \lesssim  C_\delta \eta_\eps^{2} + \delta \eps^{\gamma} \|\nabla v_\eps\|_{L^2(0,t;L^2(\Omega_\eps))}^2 + \eps^{2\gamma} +   \|v_\eps\|_{L^2(0,t;L^2(\Omega_\eps)}^2,
     \end{align}
     where we used $\alpha + \gamma \geq 3$ in the last inequality.
     
Finally, we estimate, relying on \eqref{Bog.H^1} and \eqref{w_eps-1}
     \begin{align}
        |I_2^4| &\leq  C \eta_\eps^{2}  \eta_\eps \eps^{2 \alpha -3} +  \|v_\eps\|_{L^2(0,t;L^2(\Omega_\eps))}^2  + C_\delta \eps^{\gamma}  \eta_\eps \eps^{2 \alpha -3} + \delta \eps^{\gamma}\|\nabla v_\eps\|_{L^2(0,t;L^2(\Omega_\eps))}^2.
     \end{align}
Thus, choosing $\delta$ sufficiently small, we obtain in summary, after absorbing some higher order terms,
\begin{align} \label{I_2}
    |I_2| & \leq \frac 1 4 \eps^{\gamma} \|\nabla v_\eps\|_{L^2(0,t;L^2(\Omega_\eps))}^2 +C \|v_\eps\|_{L^2(0,t;L^2(\Omega_\eps))}^2 + C\|f_\eps - f\|_{L^2(0,t;L^2(\Omega_\eps))}^2 \\
    &+ C \left( \eta_\eps \eps^{2 \alpha -3} + \eta_\eps^{-1}\eps^{3- \gamma } + \eps^{2\gamma} + \eta_\eps^{2}\right).
\end{align}

\medskip
     
    \noindent \emph{Step 5: Conclusion:}
    Inserting the bounds for $I_1$ from \eqref{I_1} and $I_2$ from \eqref{I_2} 
    into \eqref{PDE.v_eps} yields \eqref{est.v_eps}.
    
    \medskip
    
    \noindent \emph{Step 6: Adaptations in the subcritical case:}
    Let now $\gamma > 3 -\alpha$ and let $u$  solves the the Euler equations \eqref{Euler}. There are only very little changes in the proof in this case.
    In Step 1, the only differences are that  in the PDE solved by $\check  u$, \eqref{PDE.check.u_eps} all inctances of $\eps^\gamma$
    should be replaced by $\mu_0 \eps^\gamma$ (in the critical case, we assumed $\mu_0 = 1$) and that  $(M_\eps -  w^\eps \mathcal R) u$ 
    has to be replaced by $\mu_0 \eps^{\gamma + \alpha -3} M_\eps u$ . Consequently, estimate \eqref{PDE.v_eps} still holds up to replacing all instances of $\eps^\gamma$ by $\mu_0 \eps^\gamma$ and where in the source $F_\eps$ (appearing in $I_3$)  the term $(M_\eps -  w^\eps \mathcal R) u$ is likewise replaced by  $\mu_0 \eps^{\gamma + \alpha -3} M_\eps u$. In particular, the estimates for $I_1$  in Steps 3 still apply, and all the estimates of Step 4 for $I_2$ are unaffected except for the estimate of $I_2^2$
    which now takes the form
    \begin{align}
        I_2^2 &= \mu_0 \eps^{\gamma + \alpha -3}  \int_0^t \int_{\Omega_\eps} (M_\eps u )\cdot v_\eps \dd x \dd s \\
        &=\mu_0  \eps^{\gamma + \alpha -3}  \int_0^t \int_{\Omega_\eps}  ((M_\eps - \mathcal R ) u )\cdot v_\eps \dd x \dd s + \mu_0 \eps^{\gamma + \alpha -3} \int_0^t \int_{\Omega_\eps}  (\mathcal R u )\cdot v_\eps \dd x \dd s.
    \end{align}
    Thus, we estimate with Lemma \ref{lem:M_eps}
    \begin{align}
         |I_2^2| &\leq \mu_0^2 \eps^{2\gamma +2 \alpha -6}\left( \eta_\eps^{-2} \eps^{2 \alpha} + 1\right)  +  \|v_\eps\|_{L^2((0,t)\times \Omega_\eps)}^2 +  C_\delta \mu_0 \eta_\eps^{-1} \eps^{3} \eps^{\gamma +2 \alpha -6} + \delta  \mu_0 \eps^{\gamma} \|\nabla v_\eps\|^2_{L^2((0,t)\times \Omega_\eps)} \\
         &\leq   \|v_\eps\|_{L^2((0,t)\times \Omega_\eps)}^2+ \delta  \eps^{\gamma} \|\nabla v_\eps\|^2_{L^2((0,t)\times \Omega_\eps)} + C_{\mu_0,\delta}\left( \eps^{2\gamma + 2 \alpha - 6} +    \eta_\eps^{-1}\eps^{2 \alpha + \gamma   - 3} \right)
    \end{align}
    and we obtain
\begin{align} 
    |I_2| & \leq \frac 1 4 \mu_0  \eps^{\gamma} \|\nabla v_\eps\|_{L^2(0,t;L^2(\Omega_\eps))}^2 +C \|v_\eps\|_{L^2(0,t;L^2(\Omega_\eps))}^2 + C\|(f_\eps - f)\|_{L^2(0,t;L^2(\Omega_\eps))}^2 \\
    &+ C_{\mu_0} \left( \eta_\eps \eps^{2 \alpha - 3} + \eta_\eps^{-1} \eps^{2 \alpha + \gamma   - 3}  +\eps^{2 \alpha + 2\gamma-6}  + \eps^{2\gamma} + \eta_\eps^{2} \right).
\end{align}
Combining this estimate as before with the estimates for $I_1$, \eqref{I_1}, yields \eqref{est.v_eps.Euler}.
\end{proof}

\subsection{Proof of Theorem \ref{th:Darcy} and Theorem \ref{th:Darcy.quantitative}}
\label{sec:Darcy}

In this subsection, we consider $u_\eps$ a Leray solution to \eqref{Navier.Stokes.holes.rescaled} and $u$ the solution to \eqref{Darcy}.

\begin{proof}[Proof of Theorem \ref{th:Darcy.quantitative}]
    We follow closely the proof of Proposition \ref{pro:v_eps} to obtain an estimate for $v_\eps = \check u_\eps - u_\eps$, where  $\check u_\eps := w^\eps u - \mathcal B_\eps(u)$ with $w^\eps$ as in Section \ref{sec:w} and with $\mathcal B_\eps$ as in Lemma \ref{lem:Bogovski}.
    
    Recall that $w^\eps$ depends on a parameter $\eta_\eps$. We take $\eta_\eps = \eps^\beta$ for some $1 \leq \beta \leq \alpha$ to be chosen later.
    \medskip
    
\noindent\emph{Step 1: PDE solved by $\check u_\eps$:} We have  $\check u_\eps = 0$ on $(0,T) \times \partial \Omega_\eps$, and, in $(0,T) \times \Omega_\eps$
\begin{align}
    \eps^{6 - 2 \alpha - 2\gamma} \partial_t \check u_\eps - \eps^{3-\alpha} \Delta \check u_\eps  &= f - \nabla p + (M_\eps  - \mathcal R ) u  + \eps^{6 - 2 \alpha - 2\gamma} \partial_t \check u_\eps \\
    &- \eps^{3-\alpha} \nabla q^\eps u - 2 \eps^{3-\alpha} \nabla w^\eps \nabla u - \eps^{3-\alpha} w^\eps \Delta u  + \eps^{3-\alpha} \Delta \mathcal B_\eps(u),
 \end{align}
with $M_\eps$ as in \eqref{decomposition.M.gamma}. Moreover, $\dv \check u_\eps = 0$.

\medskip 

\noindent\emph{Step 2: Relative energy inequality:} Thanks to the energy inequality \eqref{energy.inequality.Darcy} as well as the PDEs solved by $u_\eps$ and $\check u_\eps$, we have, correspondingly to \eqref{energy.v_eps},
\begin{align}
  & \frac  {\eps^{6 - 2 \alpha - 2\gamma}} 2 \|v_\eps(T)\|_{L^2(\Omega_\eps)}^2 +  \eps^{3-\alpha} \int_0^T \|\nabla v_\eps\|_{L^2(\Omega_\eps)}^2 \dd t \\
    & \qquad \leq \frac {\eps^{6-2\alpha-2 \gamma}} 2   \|v^\eps(0)\|^2_{L^2(\Omega_\eps)} + \int_0^T \int_{\Omega_\eps} (F_\eps + f_\eps - f) \cdot v_\eps \dd x \dd t \\
    & \qquad \qquad+ \eps^{6-2\alpha-2 \gamma} \int_0^T \int_{\Omega_\eps} \left((u_\eps \cdot \nabla  u_\eps) \cdot \check u_\eps + \partial_t \check u_\eps \cdot v_\eps \right) \dd x \dd t,
\end{align}
where
\begin{align}
    F_\eps = (M_\eps  - \mathcal R ) u  - \eps^{3-\alpha} \nabla q^\eps u - 2 \eps^{3-\alpha} \nabla w^\eps \nabla u - \eps^{3-\alpha} w^\eps \Delta u  + \eps^{3-\alpha} \Delta \mathcal B_\eps(u).
\end{align}
Thus, using the Poincar\'e inequality \eqref{Poincare.Omega_eps} and Young's inequality,
\begin{align} \label{modulated.energy.Darcy}
    \frac 1 2 \eps^{3-\alpha}\|\nabla v_\eps\|_{L^2(0,T;L^2(\Omega_\eps))}^2  &\leq \eps^{6-2\alpha-2 \gamma} \|v_\eps(0)\|^2_{L^2(\Omega_\eps)} + \|f_\eps - f\|_{L^2(0,T;L^2(\Omega_\eps))}^2  +   |I_1| + |I_2|, \qquad  \\
    I_1 &=  \int_0^T \langle F_\eps, v_\eps \rangle \dd t, \\
    I_2 &= \eps^{6-2\alpha-2 \gamma} \int_0^T \int_{\Omega_\eps} \left((u_\eps \cdot \nabla  u_\eps) \cdot \check u_\eps + \partial_t \check u_\eps \cdot v_\eps\right) \dd x \dd t.
\end{align}

\medskip 

\noindent\emph{Step 3: Estimate of $I_1$:}
We estimate with Lemma \ref{lem:M_eps} and the Poincar\'e inequality \eqref{Poincare.Omega_eps}
\begin{align}
    \left |\int_0^T \langle (M_\eps -  \mathcal R)u, v_\eps \rangle \dd t \right| 
   &\lesssim  \eps^{\alpha - \beta}\|v_\eps\|_{L^2(0,T;L^2(\Omega_\eps)})  + \eps^{\frac{3 - \beta}2} \|v_\eps\|_{L^2(0,T;H^1(\Omega_\eps))}  \\
    &\lesssim  \left( \eps^{\frac{3-\alpha}2}  \eps^{\alpha - \beta} + \eps^{\frac{3 - \beta}2} \right) \|\nabla v_\eps\|_{L^2(0,T;L^2(\Omega_\eps))} \\
    &\lesssim \eps^{\frac{3 - \beta}2}  \|\nabla v_\eps\|_{L^2(0,T;L^2(\Omega_\eps))},
\end{align}
where we used $\alpha \geq \beta$ in the last inequality.
Moreover, since $\dv v_\eps = 0$ and using \eqref{wk.qk.phi.1} and \eqref{Hardy},
\begin{align}
    \left| \eps^{3-\alpha}  \int_0^T \int_{\Omega_\eps} (v_\eps \cdot \nabla q^\eps) \cdot u  \dd x \dd t\right| &= \left| \eps^{3-\alpha}  \int_0^T \int_{\Omega_\eps}q^\eps \cdot( v_\eps \cdot   \nabla u)  \dd x \dd t\right| \\
    &\lesssim \eps^{3-\alpha} \eps^{\frac \beta 2} \eps^{\frac{\alpha + \beta -3}{2}} \|\nabla v_\eps\|_{L^2(0,T;L^2(\Omega_\eps))} 
    = \eps^{\frac{3-\alpha}2} \eps^{ \beta} \|\nabla v_\eps\|_{L^2(0,T;L^2(\Omega_\eps))},
\end{align}
and similarly
\begin{align}
    \left| \eps^{3-\alpha}  \int_0^T \int_{\Omega_\eps} v_\eps \cdot ( \nabla w^\eps \nabla u + \eps^{3-\alpha} w^\eps \Delta u)  \dd x \dd t\right| &\lesssim \eps^{\frac{3-\alpha}2} \eps^{ \beta} \|\nabla v_\eps\|_{L^2(0,T;L^2(\Omega_\eps))} +  \eps^{3 - \alpha}  \| v_\eps\|_{L^2(0,T;L^2(\Omega_\eps))} \\
    &\lesssim \left(\eps^{\frac{3-\alpha}2} \eps^{ \beta}  +\eps^{\frac{9-3\alpha}2} \right) \|\nabla v_\eps\|_{L^2(0,T;L^2(\Omega_\eps))}.
\end{align}
Finally, by \eqref{Bog.H^1} and \eqref{w_eps-1}
\begin{align}
    \left| \eps^{3-\alpha}  \int_0^T \int_{\Omega_\eps} \nabla v_\eps : \nabla \mathcal B_\eps (u)   \dd x \dd t\right| \lesssim \eps^{3-\alpha} \eps^{\alpha - \frac{3-\beta}2} \|\nabla v_\eps\|_{L^2(0,T;L^2(\Omega_\eps))} = \eps^{\frac{3+\beta}2} \|\nabla v_\eps\|_{L^2(0,T;L^2(\Omega_\eps))}.
\end{align}
Since $\alpha \geq \beta \geq 1$ and $\alpha < 3$, we observe that $\eps^{\frac{3+\beta}2} \lesssim \eps^{\frac{3-\alpha}2} \eps^{ \beta} \lesssim  \eps^{\frac{3 - \beta}2}$ to conclude
\begin{align} \label{I_1.Darcy}
    |I_1| \leq  C \left(\eps^{\frac{3 - \beta}2} + \eps^{\frac{9-3\alpha}2}\right) \|\nabla v_\eps\|_{L^2(0,T;L^2(\Omega_\eps))} \leq\frac 1 8 \eps^{3 - \alpha} \|\nabla v_\eps\|^2 _{L^2(0,T;L^2(\Omega_\eps))} + C \left(\eps^{\alpha - \beta} +\eps^{9-3\alpha} \right). \qquad 
\end{align}

\medskip

\noindent\emph{Step 4: Estimate of $I_2$:} 
Using the identity \eqref{manipulate.trilinear} that still holds  since $u_\eps = \check u_\eps = 0$ on $\partial \Omega_\eps$ and $\dv u = \dv u_\eps = 0$, we can decompose
\begin{align}
    I_2 &= \eps^{6-2\alpha-2 \gamma} \int_0^T \int_{\Omega_\eps}( v_\eps \cdot \nabla \check u_\eps) \cdot v_\eps \dd x \dd t    +  \eps^{6-2\alpha-2 \gamma} \int_0^T \int_{\Omega_\eps} \left((\check u_\eps \cdot \nabla  \check u_\eps) \cdot  v_\eps + \partial_t \check u_\eps \cdot v_\eps\right) \dd x \dd t \\
    &=: I_2^1 + I_2^2
\end{align}
Combining the estimate \eqref{I_1^1} with the Poincaré inequality \eqref{Poincare.Omega_eps}, we have
\begin{align} 
      |I_2^1| &\lesssim \eps^{6-2\alpha -2 \gamma}  \left(\eps^{3 - \alpha} + \eps^\beta \right)   \|\nabla v_\eps\|_{L^2(0,T;L^2(\Omega_\eps))}^2.
\end{align}
Moreover, we estimate using again \eqref{Poincare.Omega_eps} as well as \eqref{Hardy}, \eqref{wk.qk.phi} and \eqref{Bog.H^1} combined with \eqref{w_eps-1}
\begin{align}
    |I_2^2| &\leq  C \eps^{6-2\alpha-2 \gamma} \int_0^T \left( \|v_\eps\|_{L^2(\Omega_\eps)} + \||\nabla w^\eps|^{\frac12} v_\eps\|_{L^2(\Omega_\eps)} \||\nabla w^\eps|^{\frac12} \check u_\eps\|_{L^2(\Omega_\eps)}  \right.\\
    &\qquad \qquad \left.+ \|\nabla v_\eps\|_{L^2(\Omega_\eps)} \|\mathcal B_\eps(u)\|_{L^2(\Omega_\eps)} \right) \dd t \\ 
    & \leq C \eps^{6-2\alpha-2 \gamma} \int_0^T \left( \eps^{\frac{3 - \alpha}2} \|\nabla v_\eps\|_{L^2(\Omega_\eps)}^2 + \eps^\beta \|\nabla v_\eps\|_{L^2(\Omega_\eps)} \|\nabla \check u_\eps\|_{L^2(\Omega_\eps)} + \|\nabla v_\eps\|_{L^2(\Omega_\eps)} \eps^\beta \eps^{\beta + \alpha - \frac 3 2} \right) \dd t\\
    &\leq \frac 1 8 \eps^{3 - \alpha }\|\nabla v_\eps\|_{L^2(0,T;L^2(\Omega_\eps))}^2 + C\left( \eps^{9-3\alpha-4 \gamma} \eps^{2\beta} \eps^{\alpha -3}  + \eps^{12 - 4 \alpha -4 \gamma}\right) 
\end{align}
Combining these estimates yields
\begin{align} \label{I_2.Darcy}
     |I_2| \leq  \eps^{3 - \alpha}\left(C \eps^{6-2\alpha-2 \gamma}+ C \eps^{3-\alpha - 2 \gamma + \beta} +  \frac 1 8 \right) \|\nabla v_\eps\|_{L^2(0,T;L^2(\Omega_\eps))}^2 + C\left( \eps^{6-2\alpha-4 \gamma + 2 \beta}+ \eps^{12 - 4 \alpha -4 \gamma}\right) \qquad 
\end{align}
\medskip

\noindent \emph{Step 5: Conclusion:} Inserting \eqref{I_1.Darcy} and  \eqref{I_2.Darcy}
 into \eqref{modulated.energy.Darcy}  yields
 \begin{align}
    \eps^{3 - \alpha} &\left(\frac 1 4 - C \eps^{6-2\alpha-2 \gamma}- C \eps^{3-\alpha - 2 \gamma + \beta} \right) \|\nabla v_\eps\|_{L^2(0,T;L^2(\Omega_\eps))}^2 \\  &\lesssim \eps^{6-2\alpha-2 \gamma} \|v_\eps(0)\|^2_{L^2(\Omega_\eps)} + \|f_\eps - f\|_{L^2(0,T;L^2(\Omega_\eps))}^2 
    +  \eps^{6 -2 \alpha - 4 \gamma + 2 \beta}  +  \eps^{\alpha - \beta} + \eps^{9-3\alpha} + \eps^{12 - 4 \alpha -4 \gamma}.
 \end{align}
We choose
\begin{align}
    \beta= \max\left\{1,\alpha - \frac{ 6- 4\gamma}{3}\right\}.
\end{align} 
Then, for all $\eps$ sufficiently small, using the assumptions $\gamma < 3/2$ and $\alpha + \gamma < 3$, the left-hand side is positive and, combination with the Poincaré inequality \eqref{Poincare.Omega_eps} yields
 \begin{align}
   \| v_\eps\|_{L^2(0,T;L^2(\Omega_\eps))}^2  &\lesssim \eps^{6-2\alpha-2 \gamma} \|v_\eps(0)\|^2_{L^2(\Omega_\eps)} + \|f_\eps - f\|_{L^2(0,T;L^2(\Omega_\eps))}^2 \\
   &+ \eps^{\frac{6 - 4\gamma}{3}} + \eps^{\alpha - 1} + \eps^{9-3\alpha} + \eps^{12 - 4 \alpha -4 \gamma}.
 \end{align}
 
Applying \eqref{v_eps.to.u_eps.2} and observing that this only produces a higher order error since $2 \alpha + \beta - 3 \geq \alpha - \beta$ thanks to $\alpha \geq \beta \geq 1$, we find
\begin{align}
         \|u_\eps - u\|_{L^2((0,T) \times \Omega_\eps)}^2  &\lesssim \eps^{6-2\alpha-2 \gamma} \|u_\eps^0 - u_0\|^2_{L^2(\Omega_\eps)} + \|f_\eps - f\|_{L^2((0,T) \times L^2(\Omega_\eps)}^2 \\
         &+ \eps^{\frac{6 - 4\gamma}{3}} + \eps^{\alpha - 1} + \eps^{9-3\alpha} + \eps^{12 - 4 \alpha -4 \gamma}.
\end{align}
This concludes the proof.
\end{proof}

\begin{proof}[Proof of Theorem \ref{th:Darcy}] For simplicity of the notation, we write $u_\eps$ instead of $\tilde u_\eps$ for the extension of $u_\eps$ by $0$ to $\R^3$. Note that the energy inequality \eqref{energy.inequality.Darcy} does not immediately provide uniform a priori estimates for $u_\eps$. The first step of the proof therefore consists in  combining the energy inequality with the  Poincar\'e inequality from Proposition \ref{prop:Poincare} to deduce a uniform a priori bound for $u_\eps$ in $L^2(0,T;L^2(\R^3)$. Then, $u_\eps \wto u$ for some $u \in L^2(0,T;L^2(\R^3))$ along subsequences and it suffices to show that $u$ solves \eqref{Darcy}.

\medskip

\noindent \emph{Step 1: Uniform  a priori estimate}
We claim that,
\begin{align} \label{a.priori.Darcy}
    \|u_\eps\|_{L^2(0,T;L^2(\R^3)} + \eps^{\frac{3-\alpha}2} \|\nabla u_\eps\|_{L^2(0,T;L^2(\R^3)} \lesssim  \eps^{3-\alpha} \eps^{-\gamma} \|u_0^\eps\|_{L^2( \R^3)} + \|f_\eps\|_{L^2(0,T;L^2(\R^3)} \lesssim 1.
\end{align}
    By the energy inequality \eqref{energy.inequality.Darcy} and the Poincar\'e inequality \eqref{Poincare.Omega_eps} we have 
\begin{align}
     \|u_\eps(t)\|^2_{L^2(\Omega_\eps)} + \eps^{2\gamma} \eps^{\alpha-3} \|\nabla u_\eps\|^2_{L^2(0,t;L^2(\Omega_\eps))} \lesssim  \|u_0^\eps\|_{L^2(\Omega_\eps)}^2   +  \eps^{2 \gamma} \eps^{\frac{3 \alpha - 9}{2}} \| f_\eps\|_{L^2(0,T;L^2(\R^3)} \| \nabla u_\eps \|_{L^2(0,T;L^2(\R^3)}.
\end{align}
   Applying Young's inequality, this establishes the estimate for $\nabla u_\eps$, and the estimate for $u_\eps$  follows by another application of the Poincar\'e inequality \eqref{Poincare.Omega_eps}.

\medskip

\noindent \emph{Step 2: Testing with $w^\eps \varphi - \mathcal B_\eps(\varphi)$:}
Let $\varphi \in C_c^\infty((0,T) \times \R^3)$ with $\dv \varphi = 0$. Then, we test the equation \eqref{Navier.Stokes.holes.rescaled} of $u_\eps$ with 
\begin{align}
    \varphi_\eps := w^\eps \varphi - \mathcal B_\eps(\varphi),
\end{align}
where $w^\eps$ is as in Section \ref{sec:w} and depends on a parameter $\eta_\eps$ which we  take as $\eta_\eps = \eps^\beta$ for some $1 \leq \beta < \alpha$ to be chosen late.
This yields
\begin{align}
   \eps^{3-\alpha} \int_0^T \int_{\R^3}  \nabla u_\eps : \nabla \varphi_\eps \dd x \dd t&=  \int_0^T \int_{\R^3} f_\eps \cdot \varphi_\eps \dd x \dd t \\
   &+  \eps^{6-2\alpha -2 \gamma} \int_0^T \int_{\R^3} \left(u_\eps \cdot \partial_t \varphi_\eps + u_\eps \cdot( u_\eps \cdot \nabla  \varphi_\eps)  \right)  \dd x \dd t.
\end{align}
It remains to show
\begin{align}
    I_1 &:= \int_0^T \int_{\R^3} f_\eps \cdot \varphi_\eps \dd x \dd t  \to \int_0^T \int_{\R^3} f \cdot \varphi \dd x \dd t, \\
        I_2 &:= \eps^{6-2\alpha -2 \gamma} \int_0^T \int_{\R^3} \left(u_\eps \cdot \partial_t \varphi_\eps + u_\eps \cdot( u_\eps\cdot\nabla  \varphi_\eps)  \right)  \dd x \dd t \to 0, \\
    I_3 &:=  \eps^{3-\alpha} \int_0^T \int_{\R^3}  \nabla u_\eps : \nabla \varphi_\eps \dd x \dd t \to  \int \mathcal R u \cdot \varphi. 
\end{align}

\medskip 

\noindent \emph{Step 2: Convergence of $I_1$:}
Recalling the assumption that $f_\eps \wto f$ in $L^2(0,T;L^2(\R^3)))$ and that $w_\eps \to \Id$ strongly in $L^2(\supp \varphi)$ by Lemma \ref{lem:w^eps} \ref{item:L^2}, we have
\begin{align}
    \int_0^T \int_{\R^3} f_\eps \cdot(w^\eps  \varphi) \dd x \dd s  \to \int_0^T \int_{\R^3} f \cdot \varphi \dd x \dd s.
\end{align}
Moreover, by \eqref{w_eps-1} and \eqref{Bog.H^1} 
\begin{align}
    \left|\int_0^T \int_{\R^3} f_\eps \mathcal B_\eps(\varphi)  \dd x \dd t \right| \lesssim \eps^\beta \eps^{\alpha - \frac{3 - \beta}2} = \eps^\alpha \eps^{3 \frac {\beta -1}{2}} \to 0
\end{align}
as $\beta \geq 1$.

\medskip

\noindent \emph{Step 3: Convergence of $I_2$:}

We have by the regularity of $u$, using \eqref{Hardy}, the a priori estimate \eqref{a.priori.Darcy} and the estimates \eqref{Bog.H^1}, \eqref{w_eps-1} and \eqref{w_eps-1.p=3}
\begin{align} 
      |I_2| &\lesssim \eps^{6-2\alpha -2 \gamma}  \int_0^t \Bigl(\|u_\eps\|_{L^2(\Omega_\eps)} (\|\partial_t \varphi \|_{L^2(\Omega_\eps)} +  \|\mathcal B_\eps(\partial_t \varphi) \|_{L^2(\Omega_\eps)})  + \|u_\eps\|_{L^6(\Omega_\eps)} \|\nabla u_\eps\|_{L^2(\Omega_\eps)} \|\mathcal B_\eps(\varphi)\|_{L^3(\Omega_\eps)} \Bigr.\\ 
      & \Bigl.\qquad \qquad \qquad  + \|u_\eps\|_{L^2(\Omega_\eps)}^2 \|\nabla \varphi \|_{L^\infty(\Omega_\eps)}  + \||u_\eps|^2 \nabla w^\eps\|_{L^1(\Omega_\eps)} \|\varphi \|_{L^\infty(\Omega_\eps)}  \Bigr) \dd x \dd s \\
    & \lesssim  \eps^{6-2\alpha -2 \gamma} \left(1  + \eps^{\beta} \eps^{\alpha + \frac{\beta-3}2} + \eps^{\alpha-3} \eps^\beta  (1+\eps^{\alpha -1} |\log \eps|^{\frac 1 3}) \right) \\
    & \lesssim  \eps^{6-2\alpha-2 \gamma} + \eps^{3 + \beta-\alpha-2 \gamma} .
\end{align}
Thanks to the assumption $\alpha > 1 $,  $\gamma < 3 -\alpha$ and $\gamma < 3/2$, we may choose $\beta \geq 1$ such that 
$ \beta \in (\alpha + 2\gamma - 3,\alpha$,
which implies $I_2 \to 0$ as $\eps \to 0$. 
\medskip

\noindent \emph{Step 4: Convergence of $I_3$:}
With $M_\eps$ as in Lemma \ref{lem:M_eps}, we rewrite
\begin{align}
    \eps^{3-\alpha} \int_0^T \int_{\R^3} \nabla u_\eps : \nabla \varphi_\eps \dd x \dd t &=  \int_0^T \langle \varphi  M_\eps, u_\eps  \rangle \dd t + \eps^{3-\alpha} \int  (u_\eps \cdot \nabla q^\eps) \cdot \varphi \dd x \dd t  \\
    &+ \eps^{3-\alpha} \int_0^T \int_{\R^3}  \nabla u_\eps : \nabla \mathcal B_\eps (\varphi) \dd x \dd t \\
    &-  \eps^{3-\alpha} \int_0^T \int_{\R^3}   u_\eps \cdot  (2 \nabla w^\eps \nabla \varphi_\eps +  w^\eps \Delta \varphi_\eps) \dd x \dd t \\
    &=: I_3^1 + I_3^2 + I_3^3 +I_3^4 .
\end{align}
By Lemma \ref{lem:M_eps} and  \eqref{a.priori.Darcy}, we have
\begin{align}
    \left|I_3^1 - \int_0^T \int_{\R^3} \mathcal R u \cdot \varphi\right| &\lesssim \left(\eps^{\alpha - \beta} \|u_\eps\|_{L^2(0,T;L^2(\R^3))} + \eps^{\frac{3 - \beta}{2}}\|\nabla u_\eps\|_{L^2(0,T;L^2(\R^3))} \right) \|\varphi\|_{L^2(0,T;H^2(\R^3))} \\
    &\lesssim \eps^{\alpha - \beta}  + \eps^{\frac{\alpha - 3}2} \eps^{\frac{3 - \beta}{2}}
    \to 0
\end{align}
since $ \beta < \alpha$.
Moreover, we estimate using \eqref{wk.qk.phi}
\begin{align}
    |I_3^2| = \left| \eps^{3-\alpha} \int_0^T \int_{\R^3}   (q^\eps  \cdot \nabla \varphi) \cdot u_\eps \dd x \dd t \right| \lesssim \eps^{\frac{3-\alpha}2} \to 0.
 \end{align}
Furthermore, by Lemma \eqref{a.priori.Darcy} and \eqref{Bog.H^1} and \eqref{w_eps-1}
\begin{align}
   |I_3^3| \lesssim \eps^{\frac{3-\alpha}2} \eps^{\alpha - \frac{3-\beta}{2}} = \eps^{\frac{\alpha + \beta}{2}} \to 0.
\end{align}
Finally, by \eqref{wk.qk.phi} and \eqref{nabla.w.infty}
\begin{align}
    |I_3^4| \lesssim \eps^{\frac{3-\alpha}2} \to 0.
\end{align}
Therefore, the desired convergence of $I_3$ is established and this finishes the proof.
\end{proof}

\section*{Acknowledgements}

The author warmly thanks
 David G\'erard-Varet for discussions that have led to several improvements of the results.

The author has been supported  by the German National Academy of Science Leopoldina, grant LPDS 2020-10.

\printbibliography
\end{document}